\documentclass{amsart}
\usepackage{amssymb,fancybox}

\newcommand{\delete}[1]{ }

\newcommand{\ftransf}{$f$-transformation}

\newcommand{\toP}{\xrightarrow{P}}

\newcommand{\la}{\lambda}
\newcommand{\eps}{\varepsilon}

\usepackage{dsfont}
\newcommand{\RR}{\mathds{R}}

\newcommand{\NN}{\mathds{N}}

\newcommand{\comment}[1]{ }
\newcommand{\longcomment}[1]{ }

\newcommand{\be}{\begin{equation}}
\newcommand{\ee}{\end{equation}}

\def\E{{\mathbb E}}
\def\V{{\mathbb V}}
\def\F{{\mathcal F}}

\def\bV{{\mathbf V}}
\usepackage{shadethm,color}

      \newtheorem{theorem}{Theorem}[section]
       \newtheorem{proposition}[theorem]{Proposition}
       
       \newtheorem{lemma}[theorem]{Lemma}

\theoremstyle{remark}
       \newtheorem{remark}[theorem]{Remark}
         
\theoremstyle{definition}
\newtheorem{definition}[theorem]{Definition}
\newtheorem{conjecture}[theorem]{Conjecture}

\newtheorem{example}{Example}[section]

\def\E{{\mathbb E}}
\def\V{{\rm Var}}
\newcommand{\Var}{\V}
\newcommand{\Cov}{{\rm Cov}}
\def\tr{{\rm tr}}

\newcommand{\calF}{\mathcal{F}}

\numberwithin{equation}{section}

\newcommand{\X}{\mathbf{X}}
\newcommand{\Y}{\mathbf{Y}}
\newcommand{\Z}{\mathbf{Z}}
\newcommand{\un}[1]{\underline{\mathbf{#1}}}

\newcommand{\wekt}[2]{\left[\begin{matrix}
#1 \\
#2
\end{matrix}\right]}

\newcommand{\wektT}[2]{\left[\begin{matrix}
#1 , #2
\end{matrix}\right]^T}
\newcommand{\macierz}[4]{\left[\begin{matrix}
  #1 & #2 \\ #3 & #4
\end{matrix}\right]}

\newcommand{\1}{\mathbf{1}}

\newcommand{\TT}{\mathcal{T}}
\newcommand{\ST}{\mathcal{S}}

\def\t{\un{t}}
\def\s{\un{s}}

\def\u{\un{u}}
\def\var{\mathbb{V}\mathrm{ar}}
\def\V{\var}
\def\Var{\var}
\def\cov{\mathbb{C}\mathrm{ov}}

\def\Var{\var}
\def\Cov{\cov}

\author{
W{\l}odek  Bryc
}

\address{
Department of Mathematics,
University of Cincinnati,
PO Box 210025,
Cincinnati, OH 45221--0025, USA}
\email{Wlodzimierz.Bryc@UC.edu}

\author{Jacek Weso{\l}owski}
\address{ Faculty of Mathematics and Information Science\\
Warsaw University of Technology\\ pl. Politechniki 1\\ 00-661
Warszawa, Poland}
\email{wesolo@alpha.mini.pw.edu.pl}


\usepackage{xcolor}
\usepackage{framed}
\colorlet{shadecolor}{gray!20}

\title{
Bridges of quadratic harnesses
}

\begin{document}
\begin{abstract}
Quadratic harnesses are typically non-homogeneous Markov processes with time-dependent state space. Motivated by a question raised in \cite[(4.4)]{Emery:2004} we give explicit formulas for bridges of such processes.
Using an appropriately defined \ftransf\  we show that  all bridges of a given quadratic harness can be transformed into other standard quadratic harnesses.
Conversely, each such bridge is an \ftransf\  of a standard  quadratic harness.
We describe quadratic harnesses that correspond to
  bridges of some L\'evy processes. We  determine all quadratic harnesses that may arise from stitching together  a pair of $q$-Meixner processes.

\end{abstract}

\maketitle

\longcomment{ Jacka oznaczenie na mosty:
$Z^{(r,a)\to(v,b)}$

czyli my robimy

$(Z^{(r,a)\to(v,b)})^f$

}
\section{Introduction}
The celebrated Paul L\'evy's Brownian bridge   of the Wiener process $(W_t)$   is $X_t=W_t-tW_1$, $t\in[0,1]$. It is clear that $X_0=X_1=0$, that the trajectories are continuous, and that
 $(X_t)$ is a Gaussian process with mean zero and covariance $E(X_sX_t)=s(1-t)$ for $0\leq s\leq t\leq 1$.
A well-known transformation (see, e.g., \cite[pg 68]{Billingsley:1968})
\begin{equation}\label{BB2W}
Y_t=(1+t)X_{t/(1+t)}
\end{equation}
converts the Brownian bridge into another Wiener process.

 In this paper we extend this representation to bridges of  a class of processes with linear regressions and quadratic conditional variances which %
we
 call  quadratic harnesses, using  an {\em \ftransf} of a stochastic process introduced 
 in Definition \ref{Def:f-transf}.
 Our main result, Theorem \ref{P B2},  shows how to transform a bridge of a quadratic harness   into another
 quadratic harness   on $(0,\infty)$. The inverse of this \ftransf\ gives an explicit formula for the bridge in terms of a quadratic harness; the formula is more involved but similar in spirit to \cite[(4.4)]{Emery:2004}.
 It is conjectured, and confirmed in many specific instances, that such quadratic harnesses are determined uniquely.
We give two  applications of Theorem \ref{P B2}: in Section \ref{Sect:conditioning} we describe quadratic harnesses that arise as \ftransf s of bridges
  of Meixner processes.  In Section \ref{Sec:gluing} we determine parameters of all  quadratic harnesses that may arise from a  construction
   that stitches together pairs of Meixner processes with random parameters.
Theorem \ref{P-chi2QH}  shows how the parameters of a quadratic harness  change under  group action of the affine transformations.
Section \ref{sect:proofs} contains  more technical proofs.

\subsection{Bridges}A heuristic description of a bridge   of a  process $(Z_t)$ is that  it
behaves like $(Z_t)$  conditioned to start at time $r$ at a prescribed point $z_r$  and to
end at time $v$  at a prescribed point $z_v$.
We formalize this intuition  using finite-dimensional distributions.

For  $t_1<t_2<\dots<t_m$ in an open  interval $\TT\subset(-\infty,\infty)$, denote $\un{Z}_{\mathbf t}=(Z_{t_1},\dots,Z_{t_m})$.
 \begin{definition}\label{Def-Bridge} A stochastic process $(X_t)_{t\in(r,v)}$ %
 is a bridge between points
 $(r,z_r)$ and $(v,z_v)$ of a process $(Z_t)_{t\in \TT}$ such that $(r,v)\subset \TT$,
  if:
 \begin{enumerate}
   \item $\lim_{t\to r+}X_t=z_r$ and  $\lim_{t\to v-}X_t=z_v$  in probability.
   \item
For   $r<t_1<t_2<\dots<t_m<v$    the distribution of the vector $\un{X}_{\bf t}$ is absolutely continuous with respect to the distribution of the vector $\un{Z}_{\bf t}$.
   \item
For   $r<s_1<\dots<s_k<t_1<t_2<\dots<t_m<u_1<u_2<\dots<u_n<v$  and integrable function $g$, if
\begin{equation}\label{BB-cond-lawsZ}
 \E(g(\un{Z}_{\mathbf t})|\un{Z}_{\mathbf s},\un{Z}_{\bf u})=h(\un{Z}_{\mathbf s},\un{Z}_{\bf u}),
\end{equation}
then
\begin{equation}\label{BB-cond-lawsX}
 \E(g(\un{X}_{\mathbf t})|\un{X}_{\mathbf s},\un{X}_{\bf u})=h(\un{X}_{\mathbf s},\un{X}_{\mathbf u}).
\end{equation}
(The equalities hold almost surely on the respective probability spaces.)
 \end{enumerate}
 \end{definition}
 (This definition in distribution neglects properties of trajectories that one may want to keep.)

The "standard"  construction of bridges of Markov processes based on Doob  $h$-transform and    duality is presented in \cite[Proposition 1]{Fitzsimmons:1993}.
A Feller property framework for existence of bridges appears in \cite{chaumont2009markovian}, see also \cite{Barczy:2005}.
There is also a related construction of more general "reciprocal processes" in  \cite{jamison1974reciprocal}, where one may prescribe the initial and final laws instead of the point masses. These constructions usually assume time-homogeneous Markov property, and/or existence of a $\sigma$-finite reference measure.
  However, the processes we are interested in are often not time-homogeneous, and  the densities with respect to a fixed $\sigma$-finite measure may fail to exist.
     We therefore prove the following proposition under assumptions that fit well Markov processes that are quadratic harnesses, see Remark \ref{Rem:Bridges}.

\longcomment{Jacek, Tuesday, April 12, 2011  11:36. Byc moze w sformulowaniu Prop. 1.2 pownnismy byli napisac jak ten
proces X sie konstruuje przywolujac odpowiednie wzory na $\nu_t$ i
$\nu_{s,x,t}$ z dowodu.}

\begin{proposition}\label{PEB}
 Suppose $(Z_t)_{t\in \TT}$ is a (non-homogeneous) Markov process with univariate distributions $\pi_t$. We  assume that there is a family of Borel sets $\{M_t: t\in \TT\}$ such that $\pi_t(M_t)=1$ and that $Z_t$ has transition probabilities $P_{s,t}(x,dy)$ defined for $x\in M_s$.
  We assume that
  \begin{itemize}
\item[(i)]
 transition probabilities  $P_{s,t}(x,dy)$
 are absolutely continuous with respect to the univariate laws $\pi_t$: for $x\in M_s$ and $s<t$, the transition probabilities are
 \begin{equation}\label{Pp}
 P_{s,t}(x,dy)=p(s,x;t,y) \pi_t(dy),
\end{equation}
where $(x,y)\mapsto p(s,x;t,y)$ is a measurable function $M_s\times M_t\to [0,\infty)$.
\item [(ii)]  there are  $r<v$ in $\TT$ and a pair $(z_r,z_v)\in M_r\times M_v$ such that $0<p(r,z_r;v,z_v)<\infty$, and for any $\eps>0$, we have
\begin{eqnarray}
\lim_{t\to r^+} \int_{\{y: |y-z_r|>\eps\}}p(r,z_r; t,y)p(t,y; v,z_v)\pi_t(dy)&=&0, \label{Brzeg1}
\\
\lim_{t\to v^-} \int_{\{y: |y-z_v|>\eps\}}p(r,z_r; t,y)p(t,y; v,z_v)\pi_t(dy)&=&0. \label{Brzeg2}
\end{eqnarray}
\end{itemize}
Then there is a Markov process $(X_t)_{t\in(r,v)}$  which is a bridge  between points $(r,z_r)$ and $(v,z_v)$ of the process $(Z_t)$.
\end{proposition}
The proof appears in Section \ref{ProofPEB}.

\begin{remark}\label{Rem:Bridges}
Condition (ii) of Proposition \ref{PEB} holds under the following assumptions.  Suppose that the following integrals exist:
$$a(t)= \int  yp(r,z_r; t,y)p(t,y; v,z_v)\pi_t(dy)   $$
$$b(t)=\int  (y-m_t)^2  p(r,z_r; t,y)p(t,y; v,z_v)\pi_t(dy) $$
and that with $m(t)=a(t)/p(r,z_r;v,z_v)$, $\sigma^2(t)=b(t)/p(r,z_r;v,z_v)$ we have
\begin{equation}\label{QHQHQH}
\lim_{t\to r^+} m(t)=z_r,\; \lim_{t\to v^-} m(t)=z_v,\;
\lim_{t\to r^+} \sigma^2(t)=\lim_{t\to v^-} \sigma^2(t)=0.
\end{equation}
Then \eqref{Brzeg1}  and  \eqref{Brzeg2} hold.

 In particular,  conditions  \eqref{QHQHQH} are easy to verify for quadratic harnesses, since  \eqref{LR} below implies
\begin{equation}\label{m(t)}
m(t)=  \frac{t-r}{v-r}z_v+\frac{v-t}{v-r}z_r ,
\end{equation}
and \eqref{EQ:q-Var} below implies that there is a constant $C=C(r,v,z_r,z_v)$ such that
\begin{equation}\label{sigma(t)}
\sigma^2(t)=C(v-t)(t-r).
\end{equation}
\end{remark}

\begin{remark}\label{R13} We will also want to consider one-sided "bridges" corresponding to  $v=\infty$. For a process $(Z_t)$ such that $\lim_{t\to\infty}Z_t/t=0$ in probability, we define a one sided bridge from $(r,z_r)$ as  time-inversion of the bridge between $(0,0)$ and $(1/r,z_r/r )$ of  the process $(tZ_{1/t})$.
\end{remark}

\subsection{Quadratic harnesses}\label{Sect:QH}
Throughout the paper the past-future filtration $(\calF_{s,t})$ is a family of sigma fields with $s<t$ from a nonempty open %
  interval $\TT=(T_0,T_1)\subset(-\infty,\infty)$ such that
$\calF_{r,u}\subset \calF_{s,t}$ for $r,s,t,u\in \TT$ with $r\leq s\leq t\leq u$.
We allow $T_0=-\infty$ or $T_1=\infty$.

An integrable stochastic process  $\X=\{X_t: \;{t\in \TT}\}$ such that
$X_s, X_t$ are   $\calF_{s,t}$-measurable, is called a
harness \cite{Hammersley,Mansuy-Yor-05,Williams73} on $\TT$  with respect to $(\calF_{s,t})$ if  for any $s,t,u\in \TT$ with $s<t<u$,
\begin{equation}\label{LR}
\E(X_t|\calF_{s,u})=\frac{u-t}{u-s}X_s+\frac{t-s}{u-s}X_u.
\end{equation}
All integrable L\'evy processes are harnesses with respect to their natural past-future filtration (\cite[(2.8)]{Jacod-Protter-88}); additional
examples  are mentioned after Definition \ref{Def_1}.

For a square-integrable process, a natural second-order analog of \eqref{LR}   is the additional requirement that $\Var(X_t|\calF_{s,u})$ is a quadratic function of $X_s, X_u$.
It turns out that under additional assumptions, conditional variance of such a process is given by   expressions  \eqref{EQ:q-Var} and \eqref{Ftsu} below, see \cite[Theorem 2.2]{Bryc-Matysiak-Wesolowski-04}.  This motivates the following definition which restricts the form of such quadratic functions to a slightly more general parametric expression.

\begin{definition}\label{Def_1}
We will say that a square-integrable stochastic process $\X=(X_t)_{t\in \TT}$ is a {\em quadratic harness} on $\TT$  with respect to $(\calF_{s,t})$,
if $(X_t)$ is a harness, and there are six constants $\chi,\eta,\theta,\sigma,\tau,\rho$ and a non-random function $F_{t,s,u}$ of $s<t<u$ such that
 \begin{multline}\label{EQ:q-Var-chi}
\V [X_t|\mathcal{F}_{s,u }]
= F_{t,s,u}\left( \chi+\eta \frac{uX_s-sX_u}{u-s} +\theta\frac{X_u-X_s}{u-s}
\right. \\
 \left.
+ \sigma
\frac{(uX_s-sX_u)^2}{(u-s)^2}+\tau\frac{(X_u-X_s)^2}{(u-s)^2}+\rho\frac{(X_u-X_s)(uX_s-sX_u)}{(u-s)^2} \right).
\end{multline}

 We will say that $(X_t)$ is a {\em standard} quadratic harness, if $\chi\ne 0$,  $\TT\subset(0,\infty)$, and
 \begin{equation}\label{EQ:cov}
\E(X_t)=0,\: \E(X_sX_t)=\min\{s,t\}.
\end{equation}
\end{definition}
Examples of quadratic harnesses on $(0,\infty)$ are five L\'evy processes  with quadratic conditional variances from \cite{Wesolowski93}.
Other examples include the classical versions of  some free L\'evy processes (\cite[Theorem 4.3]{Bryc-Wesolowski-03}), classical versions of $q$-Brownian motion (\cite[Theorem 4.1]{Bryc-Wesolowski-03}),
 bi-Poisson process 
\cite{Bryc-Matysiak-Wesolowski-04b,Bryc-Wesolowski-05}, and Markov   processes with Askey-Wilson laws (\cite[Theorem 1.1]{Bryc-Wesolowski-08}).

For standard quadratic harnesses it will be  convenient to  re-write \eqref{EQ:q-Var-chi} with  $\rho=\gamma-1$, so that
\begin{multline}\label{EQ:q-Var}
\V [X_t|\mathcal{F}_{s,u }]
= F_{t,s,u}\left( 1+\eta \frac{uX_s-sX_u}{u-s} +\theta\frac{X_u-X_s}{u-s}\right. \\
+ \sigma
\frac{(uX_s-sX_u)^2}{(u-s)^2} \left.
+\tau\frac{(X_u-X_s)^2}{(u-s)^2}
-(1-\gamma)\frac{(X_u-X_s)(uX_s-sX_u)}{(u-s)^2} \right).
\end{multline}
(This change of notation  matches better   matrix representation  in Section \ref{Sect:tranf2std} and is consistent with   \cite[Theorem 2.2]{Bryc-Matysiak-Wesolowski-04}.)
By taking the expected value of \eqref{EQ:q-Var}, see Proposition \ref{P-normalization}, we have
\begin{equation}\label{Ftsu}
F_{t,s,u}=\frac{(u-t)(t-s)}{u(1+s\sigma)+\tau-s\gamma}.
\end{equation}

When we want to indicate the parameters of a standard quadratic harness, referring to representation \eqref{EQ:q-Var} we shall write $\X \in QH(\eta,\theta;\sigma,\tau;\gamma)$.
Unless specified otherwise, we will use the natural past-future filtration  $(\calF_{s,t})$, that is
$\calF_{s,t}=\sigma(X_r: r\in\left((0,s]\cup[t,\infty)\right)\cap\TT)$.

The importance of the standard form of a quadratic harness lies in the following.
\begin{conjecture}\label{Conj}
The parameters $\eta,\theta, \sigma,\tau,\gamma$ of a standard quadratic harness
on $(0,\infty)$ determine uniquely its finite dimensional distributions.
\end{conjecture}
This conjecture is known to be true under additional, seemingly technical, conditions that include restricting the range of values of the parameters, see \cite{Bryc-Matysiak-Wesolowski-04,Bryc-Wesolowski-03,Bryc-Matysiak-Wesolowski-05,Wesolowski93}.
In particular, for $\tau\geq 0$ and $-1\leq \gamma\leq 1$, a quadratic harness in $QH(0,\theta;0,\tau;\gamma)$ is a Markov process, called $q$-Meixner process in
\cite{Bryc-Wesolowski-03};  for  the Meixner processes with $\gamma=1$, see \cite{Wesolowski93}.

We now introduce additional notation for the right hand side of \eqref{EQ:q-Var}. For a (random or deterministic) real function $\X$ of real  parameter $t$, denote
$$
\un{\Delta}_{s,u}(\X)=\wekt{\Delta_{s,u}(\X)}{ \widetilde{\Delta}_{s,u}(\X)}$$
with
\begin{equation}\label{DD}
\Delta_{s,u}(\X)=\frac{X_u-X_s}{u-s}\; \mbox{ and } \; \widetilde\Delta_{s,u}(\X)=\frac{uX_s-sX_u}{u-s}.
\end{equation}
In the sequel, if $\X$ is clear from the context,   we will  write
$\un{\Delta}_{s,u}$ instead of  $\un{\Delta}_{s,u}(\X)$.

The right hand side of  \eqref{EQ:q-Var} is a multiple of a quadratic polynomial in two real variables  which we will write  as
\begin{equation}\label{Eq:K}
K(\un{\Delta}_{s,u})= 1+\eta \widetilde{\Delta}_{s,u}+\theta\Delta_{s,u}
+ \sigma
\widetilde{\Delta}_{s,u}^2+\tau\Delta_{s,u}^2-(1-\gamma)\Delta_{s,u} \widetilde{\Delta}_{s,u}.
\end{equation}
We will write $K(a,b)$ instead of $K\left(\wekt{a}{b}\right)$.

\section{\ftransf\ of a bridge}\label{Sect:MR}

 Formula  \eqref{BB2W} describes a bridge $\X$ between points $(0,0)$ and $(1,0)$ of a quadratic harness $(W_t)\in QH(0,0;0,0;1)$ by indicating how to transform it into
 another quadratic harness
  $\bV\in QH(0,0;0,0;1)$ on $(0,\infty)$. This relation can be generalized as follows.

  Let $f:{\RR}^2\to{\RR}^2$ be a non-degenerate affine function $f:{\RR}^2\to{\RR}^2$, written in matrix notation as
\begin{equation}
  \label{f} f(x,y)=[x,y]A+\un{m}^T\;,
\end{equation}
where $\un{m}^T=[m_1,m_2]$ and
\begin{equation}
  \label{Af}
  A=\left[\begin{matrix}
  a & b \\
  c & d
\end{matrix}\right]\in GL_2(\RR).
\end{equation}
Let  $\varphi (t)=(at+b)/(ct+d)$ be the associated  M\"obius transform, $t\in\RR\setminus\{ -d/c\}$. We  note that $\varphi $ is increasing  on  subintervals of  $\RR\setminus\{ -d/c\}$ if  $\det(A)>0$, and $\varphi $ is decreasing otherwise.

\begin{definition}\label{Def:f-transf}
If  $\X=(X_t)_{t\in\TT}$ is a stochastic process on an open interval $\TT\subset \varphi(\RR\setminus\{ -d/c\})$, and $f$ is the function in \eqref{f},  we define the {\em \ftransf}  $\X^f$ of the
stochastic process $\X$ as the process $\Y=\X^f$ on  the open interval $\ST=\varphi ^{-1}(\TT)$ such that
\be\label{xf} Y_t=(ct+d)X_{\varphi (t)}+\langle \un{t},\un{m}\rangle,\; t\in \ST,  \ee
where $\langle \un{a},\un{b}\rangle=\un{a}^T\un{b}$.
\end{definition}

For example,
 formula  \eqref{BB2W} describes \ftransf\  $\mathbf{V}=\X^f$ with
 $a=c=d=1$, $b=m_1=m_2=0$. Another well known \ftransf\  is
 time inversion $Y_t=t X_{1/t}$ which can be written as $\Y=\X^f$ with $b=c=1$, $a=d=m_1=m_2=0$. Similarly, re-scaling transformation %
 $Y_t=\alpha X_{\beta t}$ with $\alpha,\beta\ne 0$
 can be written as $\Y=\X^f$ with
non-zero parameters given by %
$a=\alpha\beta$, $d=\alpha$.

A calculation verifies that as long as the time domains of the processes match,
\begin{equation}\label{fog}
(\X^f)^g= \X^{g\circ f}.
\end{equation}

This allows us to build more complicated transformations from simpler  components, and gives us the flexibility to consider either $\Y=\X^f$ or $\X=\Y^{f^{-1}}$ as needed.

By \cite[Theorem 1(i)]{Wesolowski93},  see also \cite[Theorem 4.2]{Bryc-Wesolowski-03} and \cite{Plucinska83}, if $\bV$ is a quadratic harness in $QH(0,0;0,0;1)$ on $(0,\infty)$ then $\bV$ is the Wiener process.
 So formula  \eqref{BB2W} says that if $\X$ is a Brownian bridge then an appropriate \ftransf\  $\mathbf{V}=\X^f$   gives  $\bV\in QH(0,0;0,0;1)$.
   Our next result is an analogous property of   bridges of more general quadratic harnesses.
For every bridge of a quadratic harnesses  we find an \ftransf\  turning this bridge into a standard quadratic harness on $(0,\infty)$.
\begin{theorem}
  \label{P B2}  Let $\Z=(Z_t)_{t\in\TT}$ be a standard quadratic harness with respect to its natural past-future filtration, with parameters
  specified by $\Z\in QH(\eta,\theta;\sigma,\tau;\gamma)$.
Fix $r<v$ in $\TT$  and
$z_r $, $z_v $ such that  there exits a bridge  $\X=(X_t)_{ t \in(r, v)}$ of process $\Z$
between points $(r,z_r)$ and $(v,z_v)$.
Denote
\begin{equation}\label{Deltas-Def}
\Delta_{r,v}=\frac{z_v-z_r}{v-r},\; \widetilde\Delta_{r,v}=\frac{v z_r-r z_v}{v-r},
\end{equation}
and assume that $K(\Delta_{r,v},\widetilde\Delta_{r,v})> 0$ (recall \eqref{Eq:K}).

Then  $v(1+r\sigma)+\tau-r\gamma >0$, so $\Y=\X^f$ with $f$ given by
\begin{equation}\label{AM-bridge}
A=\frac{\sqrt{v}}{(v-r)M}\macierz{1}{r}{1/v}{1},\; \un{m}^T=-\frac{\sqrt{v}}{(v-r)M}\left[\frac{z_v}{v},z_r\right],\;M^2= \frac{K(\Delta_{r,v},\widetilde\Delta_{r,v})}{v(1+r\sigma)+\tau-r\gamma}
\end{equation}
is well defined. Moreover,
  $\Y\in QH(  \widetilde \eta,\widetilde\theta;\widetilde\sigma,\widetilde\tau;\widetilde\gamma)$  on $(0,\infty)$ with
 \begin{equation}\label{theta_X}
\widetilde \eta= \frac{v\eta-\theta-2\tau \Delta_{r,v} +2\sigma v\widetilde \Delta_{r,v}   -(1-\gamma)(v \Delta_{r,v}-\widetilde \Delta_{r,v}   )}{\sqrt{v}M(v (r \sigma
   +1)+\tau-r \gamma )}\,,
\end{equation}
\begin{equation}\label{eta_X}
\widetilde \theta= \frac{\sqrt{v}\left({\theta -r\eta +2\tau  \Delta_{r,v} -2 r\sigma  \widetilde \Delta_{r,v}   -(1-\gamma )(\widetilde \Delta_{r,v}   -r \Delta_{r,v} )}\right)}{M(v (r \sigma
   +1)+\tau-r \gamma )}\,,
\end{equation}
\begin{equation}\label{Gamma_Xsigma}
\widetilde \sigma=\frac{\sigma   v^2+(1-\gamma  ) v+\tau  }{v(v
   (r \sigma  +1)+\tau-r \gamma   )}\;,
   \end{equation}
 \begin{equation}\label{Gamma_Xtau}
 \widetilde \tau=v \frac{\sigma  r^2+(1-\gamma  ) r+\tau  }{v (r \sigma
   +1)+\tau-r \gamma   }
,\end{equation}
\begin{equation}\label{Gamma_Xr}
\widetilde \gamma=\frac{v \gamma  -r (v \sigma  +1)-\tau  }{v (r \sigma
   +1)+\tau -r \gamma  }\,.
\end{equation}
\end{theorem}
The explicit formula for $\Y$ is
\begin{equation*} %
Y_t=\frac{\sqrt{v}}{(v-r)M}\left(\left(1+\tfrac{t}{v}\right)X_{(t+r)/(1+\tfrac{t}{v})}-(t\tfrac{z_v}{v}+z_r)\right).
\end{equation*}
Conversely,   we can express the bridge $\X$ in terms of a standard  quadratic harness $\Y\in QH(  \widetilde \eta,\widetilde\theta;\widetilde\sigma,\widetilde\tau;\widetilde\gamma)$  by an explicit formula
$$
X_t=a(t,v,r,z_r,z_v)Y_{v(t-r)/(v-t)}+b(t,v,r,z_r,z_v),
$$
compare \cite[Section (4.4)]{Emery:2004}. In Proposition \ref{P: bi-poisson} we describe bridges that can be obtained in this way from Poisson, negative binomial, gamma and hyperbolic secant processes.

The simplest way to handle the calculations needed for the proof of Theorem \ref{P B2} is to use matrix notation from Section \ref{Sec:MN}; the proof appears in Section \ref{Sect:proof_of_Thm1.1}.

\begin{example}
Markov process $\Z\in QH(0,0;0,0;0)$ is a classical version of the free Brownian motion, see \cite[Example 4.9]{BKS97} and  \cite[5.3]{Biane98}.
One can check that the assumptions of Proposition \ref{PEB} are satisfied, so the bridges of this process between  $(r,z_r)$ and $(v,z_v)$ exist for all $r<v$, with $z_r\in(-2\sqrt{r},2\sqrt{r})$  and  $z_v\in(-2\sqrt{v},2\sqrt{v})$. If $\X$ is such a bridge, then with $f$ from \eqref{AM-bridge}, the transformation $\X^f$ is a quadratic harness in $QH(\eta,\theta;\sigma,\tau;\gamma)$ with
$$
\eta=\frac{\widetilde\Delta_{r,v}-v\Delta_{r,v}}{v\sqrt{1-\Delta_{r,v}\widetilde\Delta_{r,v}}},
\;
\theta=\frac{r\Delta_{r,v}-\widetilde\Delta_{r,v}}{\sqrt{1-\Delta_{r,v}\widetilde\Delta_{r,v}}},
\; \sigma=\frac{1}{v}, \;\tau=r,\;\gamma=-\frac{r}{v}.
$$
These are free quadratic harnesses, $\gamma=-\sigma\tau$, without atoms from
\cite{Bryc-Matysiak-Wesolowski-05,Szpojankowski:2010}.
\end{example}

\comment{wstawka: wymaga uzgodnienia notacji}
The \ftransf\  into the "standard form" is   unique up to some special \ftransf.
\begin{proposition}  \label{Prop-ET} Suppose $f$ is given by \eqref{f} with $\un{m}=0$.
If $ \Y $ is a standard quadratic harness then $\Y^f$ is also a standard quadratic harness iff  either $b=c=0$ and  $ad=1$ or
$a=d=0$ and  $bc=1$.

In addition, if  $\Y\in QH(\eta,\theta;\sigma,\tau;\rho)$, then
\begin{equation}\label{EQ:equiv}
\Y^f \in QH(\eta/\la,\la\theta;\sigma/\la^2,\la^2\tau;\rho).
\end{equation}
with $\lambda=d$  in the first case, and
$$\Y^f \in QH(\theta/\la,\la\eta;\tau/\la^2,\la^2\sigma;\rho).$$
 with  $\lambda=c$ and swapped $\theta$, $\eta$ and  $\tau$,
$\sigma $.
\end{proposition}
\begin{proof}
 This is of course an elementary calculation, but the simplest way is to apply \eqref{cov transf} to check that
 $$\E( Y_s^f Y_t^f)= \begin{cases}
  (as +b ) (ct+d )   & \mbox{if $ad-bc>0$} \\
  (at +b ) (cs+d )
& \mbox{if $ad-bc<0$}
 \end{cases} $$
This proves the first part. The second part is again an elementary calculation that is included in the more general formula \eqref{Gamma A}.
\end{proof}
\comment{koniec wstawki}

For bridges between points $(0,0)$  and  $(v,z_v)$ and for one-sided bridges  from $(r,z_r)$, formulas for parameters   correspond to Theorem \ref{P B2} after taking $(r,z_r)=(0,0)$ or by taking the limit  as $ v\to\infty$, while keeping $z_r,z_v$ fixed.   These two situations are described in the following remark.
\begin{remark}\label{Remark onesided bridges} Suppose $\Z$ is a quadratic harness  on $(0,\infty)$ in $QH(\eta,\theta;\sigma,\tau;\gamma)$.
\begin{itemize}
\item[(i)]
Let  $\X=(X_t)_{t\in(0, v)}$ be the bridge between points $(0,0)$ and $(v,z_v)$ of the process $\Z$. Assume that
$$
\kappa^2=\left(1+\frac{\tau}{v}\right) \left(1+\theta \frac{z_v}{v}+\tau \frac{z_v^2}{v^2}\right)>0.
$$
Then process $\Y=\X^f$ with
$$A=\frac{v+\tau}{v\kappa}\macierz{1}{0}{1/v}{1} ,\; \un{m}=-\frac{v+\tau}{v\kappa}\wekt{{z_v}/{v}}{0},$$ written explicitly as
\begin{equation}\label{EQ:transform_me_again}
Y_t=\frac{v+\tau}{v\kappa} \left((1+\tfrac{t}{v}) X_{vt/(t+v)}-\tfrac{t}{v}z_v\right),
\end{equation}

is a quadratic harness with parameters
\begin{eqnarray}
   \theta_Y &=& \frac{\theta  +2\tau z_v/v}{\kappa} \label{THETA-Y}%
   \,,\\
   \eta_Y &=& \frac{v\eta-\theta-2\tau z_v/v-(1-\gamma)z_v}{v\kappa} \label{eta_XV}\,,\\
\tau_Y&=&\frac{v \tau  }{v +\tau } \label{tau_XV}\,, \label{TAU-Y}\\
\sigma_Y&=&\frac{\sigma   v^2+(1-\gamma  )v+\tau}{ v(v+\tau)} \label{sigma_XV}\,,\\
\gamma_Y&=&\frac{v \gamma  -\tau  }{v +\tau }\label{gamma_XV}\,.
\end{eqnarray}

\item[(ii)] Let  $\X=(X_t)_{ t\in(r,\infty)}$ be the one-sided bridge  from $(r,z_r)$ of process $\Z$, see Remark \ref{R13}. Assume that
$\kappa^2=(1 +r\sigma) (1+\eta z_r+\sigma z_r^2)>0$.
For $0<t<\infty$, let
$\Y=\X^f$ with
$$
A=\tfrac{1+r\sigma}{\kappa}\macierz{1}{r}{0}{1},\; \un{m}=-\tfrac{1+r\sigma}{\kappa}\wekt{0}{z_r},
$$
which can be explicitly written as
\begin{equation}\label{YYY2B}
Y_t=\tfrac{1+r\sigma}{\kappa}(X_{t+r}-z_r).
\end{equation}

Then $\Y $ is a quadratic harness with parameters
\begin{eqnarray}
   \theta_Y &=& \displaystyle\frac{\theta -r\eta -2 r\sigma  z_r-(1-\gamma )z_r}{\kappa} \label{theta_XR}\,,\\
   \eta_Y  &=& \displaystyle\frac{\eta+2\sigma z_r}{\kappa} \label{eta_XR}\,,\\
   \tau_Y&=& \frac{\sigma  r^2+(1-\gamma  ) r+\tau  }{1 +r\sigma} \label{tau_XR}\,,\\
\sigma_Y&=&\frac{\sigma  }{1 +r\sigma}
 \label{sigma_XR}\,,\\
\gamma_Y&=&\frac{\gamma  -r  \sigma  }{1 +r\sigma}  \label{gamma_XR}\,.
\end{eqnarray}

 \end{itemize}
\end{remark}

Next we present  relations between the parameters of the original quadratic harness that are preserved by the \ftransf\   \eqref{AM-bridge} of its bridge.
\begin{proposition}\label{P Cond inv}  Let $\X$ be a  bridge  between points $(r,z_r)$ and $(v,z_v)$ of a process $\Z\in QH(\eta_Z,\theta_Z;\sigma_Z,\tau_Z;\gamma_Z)$.
Let $\Y=\X^f$  with $f$ defined in  \eqref{AM-bridge}, so that  $\Y\in  QH(\eta_Y,\theta_Y;\sigma_Y,\tau_Y;\gamma_Y) $.
\begin{itemize}
\item[(i)]  Then $\gamma_Z=-1$ if and only if $\gamma_Y=-1$.
 \item[(ii)]  If $\gamma_Z>-1$, then
\begin{equation}
  \label{Cond inv}
  \frac{(1-\gamma_Y)^2-4\sigma_Y\tau_Y}{(1+\gamma_Y)^2}=\frac{(1-\gamma_Z)^2-4\sigma_Z\tau_Z}{(1+\gamma_Z)^2}.
\end{equation}
 \item[(iii)] If  $(r,z_r)=(0,0)$, then
 \begin{equation}\label{signs}
   \rm{sign}(\theta_Y^2-4\tau_Y)= \rm{sign}(\theta_Z^2-4\tau_Z)\;.
 \end{equation}
\end{itemize}
\end{proposition}
\begin{proof} The first statement follows from \eqref{Gamma_Xr}.
Formula \eqref{Cond inv} follows by a direct computation from (\ref{Gamma_Xsigma}--\ref{Gamma_Xr}).
Formula \eqref{signs} follows from  formulas  \eqref{THETA-Y} and \eqref{TAU-Y}, which give
$\theta_Y^2-4\tau_Y=(\theta_Z^2-4\tau_Z)/\kappa^2$.
\end{proof}

\subsection{Application: bridges of Meixner processes}\label{Sect:conditioning}
By the Meixner processes we mean L\'evy processes for which the univariate laws belong to the natural exponential families with quadratic variance functions, \cite{Morris:1982}. This class consists of    the Wiener, Poisson, negative binomial, gamma, and hyperbolic secant processes. Some authors use the name Meixner process for the  hyperbolic secant process only;
\cite[Section 4.3]{schoutens2000stochastic} calls the whole class the L\'evy-Meixner systems.
According to   \cite{Wesolowski93}, up to affine transformations this class coincides with the standard quadratic harnesses such that $\eta=\sigma=0$ and $\gamma=1$.

From Proposition \ref{PEB} together with Remark \ref{Rem:Bridges}, see also \cite[Proposition 1]{Fitzsimmons:1993}, it follows that bridges with endpoints in the support of the univariate laws exist for all five quadratic harnesses in $QH(0,\theta;0,\tau;1)$.
Such bridges are quadratic harnesses, and from Theorem \ref{P B2}  we read out their standard form.
It turns out that their parameters satisfy relations $0\leq \sigma\tau<1$, $\gamma=1-2\sqrt{\sigma\tau}$ and $\eta\sqrt{\tau}+\theta\sqrt{\sigma}=0$.

\begin{proposition}\label{P cond}
 Fix $\tau\geq 0$ and $\theta\in\RR$. Suppose  $\Z\in QH(0,\theta;0,\tau;1)$ i.e., $\Z$ is a Meixner process. Let $\X$ be a bridge  between points $(r,z_r)$ and $(v,z_v)$ of $\Z$, and $\Y=\X^f$ with
 $$
 A=\frac{\sqrt{v-r+\tau}}{(v-r)\sqrt{1+\theta\Delta_{r,v}+\tau\Delta_{r,v}^2}}\macierz{v}{r}{1}{1},$$
 $$
 \un{m}=-\frac{\sqrt{v-r+\tau}}{(v-r)\sqrt{1+\theta\Delta_{r,v}+\tau\Delta_{r,v}^2}}\wekt{z_v}{z_r},
 $$
 (recall \eqref{Deltas-Def}), i.e., explicitly
  \begin{equation}\label{31415926}
  Y_t=\frac{\sqrt{v-r+\tau}}{(v-r)\sqrt{1+\theta\Delta_{r,v}+\tau\Delta_{r,v}^2}}\left((1+t)X_{(vt+r)/(1+t)}- t z_v- z_r\right).
\end{equation}

 Then $\Y\in QH(\eta_Y,\theta_Y;\sigma_Y,\tau_Y;\gamma_Y)$  with parameters
 \begin{equation}
  \label{Cond meixner I}
  \gamma_Y=\frac{v-r-\tau}{v-r+\tau},\; \tau_Y=\sigma_Y=\frac{\tau}{v-r+\tau}\,,
\end{equation}
\begin{equation}
  \label{Cond meixner II}  \theta_Y =-\eta_Y=\frac{\theta+2 \tau\Delta_{r,v}}{\sqrt{v-r+\tau}\sqrt{1+\theta\Delta_{r,v}+\tau\Delta_{r,v}^2}}\,.
  \end{equation}

\end{proposition}
\begin{proof}
  From Theorem  \ref{P B2} we read out the transformation that leads to parameters
 $$\gamma'=\frac{v-r-\tau}{v-r+\tau},\; \tau'=\ \frac{\tau v}{v-r+\tau},\;
\sigma'=\frac{\tau}{v(v-r+\tau)},$$
$$
 \theta' =\sqrt{v}\frac{\theta+2 \tau\Delta_{r,v}}{\sqrt{v-r+\tau}\sqrt{1+\theta\Delta_{r,v}+\tau\Delta_{r,v}^2}},$$
 $$
 \eta'=-\frac{\theta+2 \tau\Delta_{r,v}}{\sqrt{v}\sqrt{v-r+\tau}\sqrt{1+\theta\Delta_{r,v}+\tau\Delta_{r,v}^2}}\,.$$
 Composing the \ftransf\  from  \eqref{EQ:equiv}  with the \ftransf\  from Theorem  \ref{P B2}, see \eqref{fog}, we get  \eqref{31415926} and parameters as claimed.
   \end{proof}

\longcomment{Jacek Tuesday, April 12, 2011  12:54: Brak podzielenia przez $\sigma_0\tau_0$ we wzorze na $v-r$  [poprawione].
W ogole $r$ jest tu niepotrzebne. Mozna bylo wziac $ r=0$.}

\begin{example}[Dirichlet process]\label{Example: Dirichlet harness}
For any $\sigma_0,\tau_0>0$ with $\sigma_0\tau_0<1$,   there exists a standard quadratic harness $\Y$ (namely, a Dirichlet process) on $(0,\infty)$ which has
 parameters $\sigma_Y=\sigma_0$, $\tau_Y=\tau_0$, $\theta_Y=2\sqrt{\tau_0}$, $\eta_Y=-2\sqrt{\sigma_0}$,
and $\gamma_Y=1-2\sqrt{\sigma_0\tau_0}$.

Indeed,  consider a bridge of  a  gamma process $(G_t)_{t>0}$.
The gamma process  $(G_t)$ is a non-negative   L\'evy processes with two
parameters $\alpha,\beta> 0$.
The density of $G_t$ is given by
\begin{equation}\label{G-d}
 \tfrac{\beta^{\alpha t}}{\Gamma(\alpha)} x^{\alpha t-1} e^{-\beta x}\1_{(0,\infty)}(x).
\end{equation}
As a L\'evy process,  $(G_t)$ is a harness with
mean
$\E(G_t)=t\alpha/\beta $ and variance $\Var(G_t)=t\alpha/\beta^2$.
It is also known (see \eqref{V-gamma} below) that
\begin{equation}\label{V-gammaG}
\Var(G_t|\mathcal{F}_{s,u})=\frac{(t-s)(u-t)}{(u-s)^2((u-s)/\alpha+1)}(G _u-G_s)^2\;.
\end{equation}
Then
$$
Z_t=\frac{\beta}{\alpha} G_{\alpha t} -  \alpha t
$$
is a quadratic harness in $QH(0,2/\alpha;0,1/\alpha^2;1)$ which by further scaling as in
\eqref{EQ:equiv} can be transformed into a quadratic harness in $QH(0,2;0,1;1)$.
So instead of considering bridges of $(G_t)$, we  consider bridges of $\Z\in QH(0,2;0,1;1)$. To apply Proposition \ref{P cond} we
choose $v-r=1/\sqrt{\sigma_0\tau_0}-1$ so that $1/(1+v-r)=\sqrt{\sigma_0\tau_0}$. From
\eqref{Cond meixner I} %
we have
$$
\gamma_Y=1-\frac{2}{v-r+1}=1-2\sqrt{\sigma_0\tau_0}\,.
$$
Transformation \eqref{EQ:equiv} with $\la^2=\sqrt{\sigma_0/\tau_0}$ gives
$$
\tau_Y=\sqrt{\sigma_0\tau_0}/a^2=\tau_0,\; \sigma_Y=a^2\sqrt{\sigma_0\tau_0}=\sigma_0.
$$
By   \eqref{Cond meixner II}  with any $\Delta=\Delta_{r,v}\geq 0$,
$$a\theta_Y=\frac{2(1+\Delta)}{\sqrt{ v-r+1}\sqrt{1+2\Delta+\Delta^2}}=
\frac{2}{\sqrt{ v-r+1}} = 2\sqrt[4]{\sigma_0\tau_0},$$
so $\theta_Y=2\sqrt{\tau_0}$. Similarly,
$\eta_Y/a=  -2\sqrt[4]{\sigma_0\tau_0}$, so $\eta_Y=-2\sqrt{\sigma_0}$.

Since bridges of the gamma process are Dirichlet processes,  the same conclusion can be obtained directly by a fairly natural reparameterization \eqref{X dirichlet} without invoking explicitly any of the transformations.
 Let $a_1,\ldots,a_n,a_{n+1}$ be
positive numbers. A Dirichlet distribution
is defined through its density
$$
f(x_1,\ldots,x_n)=\frac{\Gamma(a_1+\ldots+a_{n+1})}{\prod_{i=1}^{n+1}\Gamma(a_i)}\prod_{i=1}^n\:x_i^{a_i-1}
\left(1-\sum_{i=1}^nx_i\right)^{a_{n+1}}\1_{U_n}(x_1,\ldots,x_n)\;,
$$
where $U_n=\{(x_1,\ldots,x_n)\in
(0,\infty)^n:\;\sum_{i=1}^n\:x_i<1\}$. A stochastic process
$\X =(X_t)_{t\in(0,v )}$ is called a Dirichlet process if there
exists a finite nonzero measure $\mu$ on $(0,v )$ such that for any $n$ and
any $0=t_0< t_1<\ldots<t_n< v $ the distribution of the vector of
increments $(X _{t_1},X_{t_2}-X_{t_1},\ldots,X_{t_n}-X_{t_{n-1}})$
is Dirichlet
with  $a_j=\mu((t_{j-1},t_j])$, $1\leq j\le n$ and $a_{n+1}=\mu((t_n,v ))$.
This is one of the basic objects of non-parametric Bayesian
statistics - see %
 \cite{Doskum-1974,Ferguson-1973}. %
Let
$\mu=c\lambda$, where $\lambda$ is a Lebesgue measure on $(0,v )$
and $c=1/\alpha>0$ is a number. Recall that the beta distribution,
$B_I(a,b)$, is defined by the density
$$
f(x)=\frac{\Gamma(a+b)}{\Gamma(a)\Gamma(b)}x^{a-1}(1-x)^{b-1}\1_{(0,1)}(x)\,,
$$
and if $X\sim B_I(a,b)$ then \be
\label{beta}
\E(X)=\frac{a}{a+b},\quad
\mbox{and}\quad
\var(X)=\frac{ab}{(a+b)^2(a+b+1)}\;.
\ee

Since $X_t$ has the beta distribution $B_I(ct,c(v -t))$ the
formulas \eqref{beta} give
$$
\E(X_t)=\tfrac{t}{v }\quad\mbox{and}\quad
\Cov(X_s,X_t)=\tfrac{s(v -t)}{v ^2(cv+1)}\;.
$$
Note that to compute $\E(X_sX_t)$ it is convenient to use the
classical fact, that $X_s/X_t$ and $X_t$ are independent and
$X_s/X_t$ is a beta $B_I(cs,c(t-s))$ random variable. Note also
that $\X $ is a Markov process with transition distribution defined by
the fact that $({X_t-X_s})/({1-X_s})$ and $X_s$ are independent,
and $(X_t-X_s)/({1-X_s})$ is beta $B_I(c(t-s),c(v -t))$. It is
also known that
$$
\frac{X_t-X_s}{X_u-X_s}\sim B_I(c(t-s),c(u-t)).
$$
Moreover,  $(X_t-X_s)/(X_u-X_s)$ and $(X_s,X_u)$ are independent.
Therefore, from \eqref{beta} we get
$$
\E(X_t|\mathcal{F}_{s,u})=X_s+(X_u-X_s)\frac{t-s}{u-s}
$$
and thus $\X$ is a harness. The second formula in
\eqref{beta} gives
\begin{equation}\label{V-gamma}
\Var(X_t|\mathcal{F}_{s,u})=(X_u-X_s)^2\frac{(t-s)(u-t)}{(u-s)^2(c(u-s)+1)}\;.
\end{equation}
This is an example of quadratic harness with $\chi=0$.

Define now
\begin{equation}
  \label{X dirichlet}
  Y_t=\sqrt{\tfrac{1+cv}{v}}\left((v +t)X_{\frac{tv}{v +t}}-t\right),\quad
t\in(0,\infty)\;.
\end{equation}
Note that $\Y$ is an \ftransf\  of $\X$ with $f$ defined by
$$
A=\sqrt{\tfrac{1+cv}{v}}\macierz{v}{0}{1}{v},\; \un{m}=-\sqrt{\tfrac{1+cv}{v}}\wekt{1}{0}.
$$
It is elementary to check that $(Y_t)_{t\in(0,\infty)}$ is a
quadratic harness  and
the parameters are as follows
$$
\theta_Y=\frac{2\sqrt{v}}{\sqrt{1+cv}}\;,\quad\eta_Y=\frac{-2}{\sqrt{v} \sqrt{1+cv}}\;,
$$$$
\tau_Y=\frac{v}{1+cv}\;,\quad\sigma_Y=\frac{1}{v(1+cv)}\;,\quad\gamma_Y=1-\frac{2}{(1+cv)^2}\;.
$$
(Note that this agrees with the answers  deduced from Proposition \ref{P cond} which implies that $\theta_Y^2=4\tau_Y$, $\eta_Y^2=4\sigma_Y$ and
$\gamma_Y=1-2\sqrt{\sigma_Y\tau_Y}$.)

On the other hand, it can be easily seen that the process $\X$ is a
bridge of the gamma process $(G_t)_{t\in(0,\infty)}$ governed by
the gamma distribution with the shape parameter $1/c$ and the
scale equal $1$. More precisely,   process $\X$ is
equal in distribution to the gamma bridge
$(G_t/G_v )_{t\in(0,v)}|G_v \sim %
(G_t/G_v )_{t\in(0,v)}$ between points $(0,0)$ and $(v,1)$, see \cite[Definition 2]{Ferguson-1974}, see also \cite{Emery:2004}.

\end{example}

\begin{example}[Binomial process]\label{EX_Bin}
Fix  $n\in\NN$ and real $\eta_0,\theta_0$ such that  $\eta_0\theta_0=-1/n$. Then   there exist a  standard quadratic harness  $\Y$ (namely, the Binomial process described here)  on $(0,\infty)$  which has
 parameters $\sigma_Y=\tau_Y=0$, $\theta_Y=\theta_0$, $\eta_Y=\eta_0$,
and $\gamma_Y=1$.

Indeed,  consider   standard quadratic harnesses arising from bridges of a Poisson process.
 Poisson process $N_t$ with parameter $\la>0$ is a harness with mean $\E(N_t)=\la t$
 variance $\Var(N_t)=\la t$, and with
 conditional variance with respect to the natural past-future filtration given by
 $$
 \Var(N_t|\calF_{s,u})=\frac{(u-t)(t-s)}{(u-s)^2}(N_u-N_s).
 $$
 Then
 $$
 Z_t= N_{t/\la}- t
 $$
 is in $QH(0,1;0,0;1)$.
So instead of considering bridges of $(N_t)_{t>0}$ we   consider a bridge of $\Z$ between points $(0,0)$ and $(v,n)$.
From \eqref{Cond meixner I} we see that $\gamma_Y=1$, $\sigma_Y=\tau_Y=0$.
Since $\Delta_{0,v}=(n-v)/v$, from  \eqref{Cond meixner II} we see that
$\theta_Y=-\eta_Y=1/\sqrt{n}$. %
So
$\eta_Y\theta_Y=-1/n$.

The same conclusion can be obtained more directly without invoking explicitly any of the transformations. Let $b(m,p)$ denote the binomial distribution with sample size $m$ and probability of success $p$.
 For a fixed $n\in\NN$, define a Markov process
$\X=(X_t)_{t\in(0,v)}$ by the following (consistent) family of
marginal and conditional distributions:
$$
X_t\sim b\left(n,\tfrac{t}{v}\right)\quad\mbox{and}\quad
X_t-X_s|X_s\sim b\left(n-X_s,\tfrac{t-s}{v-s}\right),\;0< s\le
t<v\;.
$$
Then  process $\X$ is called a binomial process with parameter
$n$.  (Compare \cite[Proposition 4.4]{Bryc-Wesolowski-03}.)  It is elementary to see that the conditional distribution
$X_t|{\mathcal F}_{s,u}\sim
b\left(X_u-X_s,\frac{t-s}{u-s}\right)$. Therefore $\X$  is a harness i.e.  \eqref{LR} holds, and
for any
$s,t,u\in(0,v)$, $s<t<u$
$$
\Var(X_t|\mathcal{F}_{s,u})=\frac{(u-t)(t-s)}{(u-s)^2}(X_u-X_s)\;.
$$
This is again an example of quadratic harness with $\chi=0$.

Let
$$
Y_t=\frac{(v+t)X_{\frac{t v}{v+t}}-nt}{\sqrt{nv}}\;,\quad
t\in(0,\infty),
$$
i.e., $\Y$ is an \ftransf\  of $\X$ with $f$ defined by
$$
A=\tfrac{1}{\sqrt{nv}}\macierz{v}{0}{1}{v},\; \un{m}=\sqrt{\tfrac{n}{v}}\,\wekt{1}{0}.
$$
Then an easy computation  shows that  process $(Y_t)_{t> 0}$ is a quadratic harness  and  the parameters are
$\theta_Y=\sqrt{v/n}$, $\eta_Y=-1/\sqrt{ nv }$, $\tau_Y=\sigma_Y=0$ and
$\gamma_Y=1$.

On the other hand, it is immediate that $\X$ is a  bridge
obtained by conditioning a Poisson process $(N_t)_{t> 0}$ at $N_v=n$.

\end{example}

\subsection{Application: stitching construction}\label{Sec:gluing}
This section is motivated by the construction of a classical bi-Poisson process from a pair of two conditionally independent Poisson processes in \cite[Proposition 4.1]{Bryc-Wesolowski-05}, and by the construction of a quadratic harness from two conditionally independent  negative binomial  processes  in \cite[Proposition 5.1]{Maja:2009}.
These constructions essentially consist of choosing an appropriate   time $T$ and an appropriate law for random variable $Z_T$.
For $z_T$ in the support of $Z_T$, the  bridge $\X_+$ between $(0,0)$ and $(T,z_T)$ of process $\Z$
transforms into  a Poisson process (a negative binomial process)  by  \ftransf\ \eqref{EQ:transform_me_again}.
Similarly, the one-sided bridge  $\X_-$  from $(T, z_T)$  transforms into another  Poisson process
(another negative binomial process). The two Poisson processes, or the two negative binomial processes, used in the stitching construction are $Z_T$-conditionally independent.

 We use Theorem \ref{P B2} and Remark \ref{Remark onesided bridges}(i) to determine parameters of all quadratic harnesses that might arise from   such a stitching construction
 from  more general $q$-Meixner processes, which are quadratic harnesses that generalize the Meixner processes by allowing arbitrary $\gamma\in[-1,1]$, see  \cite{Bryc-Wesolowski-03}.
Namely, if a quadratic harness $\Z$ comes from stitching together two $q$-Meixner processes, then there is at least one pair $(T,z_T)$ such that the bridge $\X_-$ between $(0,0)$ and $(T,z_T)$ exists and  can be transformed back into a $q$-Meixner process $\Y$ with parameters given in Remark \ref{Remark onesided bridges}(i). The following result describes the parameters of the standard quadratic harness $\Z$ in such situation.
 \begin{proposition}\label{P: bi-poisson}
Let $\Z\in QH(\eta,\theta;\sigma,\tau;\gamma)$ be defined  on $(0,\infty)$. Suppose that there are  real numbers $T>0$ and $z_T$ such that the  bridge  $\X_-$   between points $(0,0)$ and $(T,z_T)$  of
 $\Z$, transforms by formula \eqref{EQ:transform_me_again}   into a $q$-Meixner process $\Y$. Then  $\Y$ is a Meixner process and one of the following cases must happen:
\begin{itemize}
\item[(i)]   $\gamma=1$, $\sigma=\tau=0$ and $\eta=\theta=0$;

\item[(ii)] $\gamma=1$, $\sigma=\tau=0$ and $\eta\theta>0$;
  \item[(iii)]  $\sigma,\tau> 0$,  $\gamma=1+2\sqrt{\sigma\tau}$ and  $\eta\sqrt{\tau}=\theta\sqrt{\sigma}$.  \end{itemize}
\end{proposition}
\begin{proof}

  The only possibility for \eqref{sigma_XV}  to correspond to a $q$-Meixner process is when the parameters of $\Z$ satisfy
\begin{equation}
  \label{!}\sigma T^2+(1-\gamma)T+\tau=0.
\end{equation}

Since $\sigma,\tau\geq 0$ (see \cite[Theorem 2.2]{Bryc-Matysiak-Wesolowski-04}),  the only solution with $\gamma\leq 1$ is $\gamma=1$, $\sigma=\tau=0$. Then from  \eqref{eta_XV}  we see that $\Y$ is indeed a Meixner process when we set $T=\theta/\eta$ %
 or when $T>0$ is arbitrary but $\eta=\theta=0$.

Other solutions of  \eqref{!} interpreted as a quadratic equation in $T$ are possible only when $(1-\gamma)^2\geq 4\sigma\tau$.
  However, since $\gamma\leq 1+2\sqrt{\sigma\tau}$ by \cite[Theorem 2.2]{Bryc-Matysiak-Wesolowski-04},  this gives  $\gamma= 1+2\sqrt{\sigma\tau}$ and $T=\sqrt{\tau/\sigma}$.
  Then   from    \eqref{eta_XV}, the coefficient at $z_T$ vanishes  when  $\eta\sqrt{\tau}=\theta\sqrt{\sigma}$, so $\Y$ is
   indeed a Meixner process.

\end{proof}
\begin{remark}
We expect that stitching constructions work in all of the cases described in Proposition \ref{P: bi-poisson}. Case (i) is trivial, with    $\Y$ being  the Wiener process. In case (ii)  $T=\theta/\eta$,   $\Y$ is  the Poisson processes with parameter $\lambda$ which depends on $z_T$ and the stitching construction is described in \cite[Proposition 4.1]{Bryc-Wesolowski-05}. In case (iii) with $T=\sqrt{\tau/\sigma}$, $\gamma=1+2\sqrt{\sigma\tau}$,  by  \eqref{signs} the sign of $\theta^2-4\tau$ is preserved. From \cite[Theorem 1]{Wesolowski93} we see that $\Y$ is either a negative binomial ($\theta^2>4\tau$), or a gamma ($\theta^2=4\tau$), or a hyperbolic secant  ($\theta^2<4\tau$) process. The stitching construction for the negative binomial process appears in  \cite[Proposition 5.1]{Maja:2009}.
\end{remark}

\section{Transforming quadratic harnesses into  standard form}\label{Sect:tranf2std}
Bridges of standard quadratic harnesses  are quadratic harness but they are not in the standard form because their means are affine functions of time and they have product covariances.
In our next theorem  we present \ftransf s \eqref{xf}  that convert
such quadratic harnesses  into  standard form.

\begin{theorem}\label{P-chi2QH}
Let $\X$ be a harness \eqref{LR} with respect to a past-future filtration $(\mathcal{F}_{s,t})$ on an interval $(T_0,T_1)\subset\RR$ with  mean
$$
\E(X_t)=\alpha +\beta t
$$
and with covariance
\begin{equation}\label{Cov[A]}
\Cov(X_s,X_t)=(as+b)(ct+d), \; s<t,
\end{equation}
such that $ad-bc>0$ and $(at+b)(ct+d)>0$ for $t\in(T_0,T_1)$.
 Suppose that \eqref{EQ:q-Var-chi} holds, and that
\begin{equation}\label{chi+}
\widetilde\chi:=\chi+\alpha\eta+\theta\beta+\sigma \alpha^2+\tau\beta^2+\rho\alpha\beta>0.
\end{equation}
Let $\psi(t)=(dt-b)/(a-ct)$. Then
stochastic process
 \begin{equation}\label{EQ:chi2QH}
 Y_t=\frac{a-ct}{ad-bc} \left(X_{\psi(t)}-\alpha-\beta\psi(t)\right)
\end{equation}
  is a quadratic harness in $QH(\eta',\theta';\sigma',\tau';\rho')$ on the interval $\left(\frac{a T_0+b}{cT_0+d},\frac{aT_1+b}{cT_1+d}\right)\subset(0,\infty)$, and has
 parameters
\begin{eqnarray}
\eta'&=&({d (\eta +\beta  \rho +2 \alpha
   \sigma )+c (\theta +\alpha  \rho +2 \beta
   \tau )})/\widetilde{\chi}\,,
   \label{eta'}
   \\
\theta'&=& %
({b (\eta +\beta  \rho +2 \alpha
   \sigma )+a (\theta +\alpha  \rho +2 \beta
   \tau )})/\widetilde{\chi}\,,
   \\
\sigma'&=&%
({\tau  c^2+d \rho  c+d^2 \sigma
   })/\widetilde{\chi}\,,
   \\
\tau'&=& %
({\tau  a^2+b \rho  a+b^2 \sigma
   })/\widetilde{\chi}\,,
   \\
\rho'&=&(b c \rho +a d \rho +2 b d \sigma +2
   a c \tau )/ {\widetilde{\chi}} \,.
    \label{gamma'}
\end{eqnarray}
\end{theorem}
 We  remark that the \ftransf\ used in Theorem \ref{P-chi2QH} is reversible,
 so $\X=\Y^f$ with
$$A=\macierz{a}{b}{c}{d}, \; \un{m}=\wekt{\beta}{\alpha},$$
that is
 $$X_t=(c t+d)Y_{(at+b)/(ct+d)}+\alpha+\beta t.$$

The proof  of Theorem \ref{P-chi2QH} uses matrix notation and is postponed until Section  \ref{Sec:proof_P-chi2chi}.
Here we   use Theorem \ref{P-chi2QH} to give examples which show that Conjecture \ref{Conj} does not hold on finite intervals.
\begin{example}\label{Ex:W+tX} Let  $(W_t)_{t> 0}$ be the Wiener process and $\xi$ be a centered random variable independent of $W$ with $\E\xi^2=v^2$.
Let
$$
X_t=W_t+\xi  t, \; t>0.
$$
Then $E(X_t)=0$ and $\Cov(X_s,X_t)=s(1+v^2t)$.
Furthermore, $X_t$ is a harness with respect to its natural past-future filtration, and
$$\Var(X_t|\mathcal{F}_{s,u})=\Var(W_t|W_s,W_u)=F_{t,s,u}.$$
 So from Theorem \ref{P-chi2QH} (or by direct calculation) we see that
$$ %
Y_t=(1-tv^2)X_{t/(1-tv^2)}
$$
is a standard quadratic harness on $(0, 1/v^2)$   and has parameters $\eta=\theta=\sigma=\tau=0$, $\gamma=1$.
\end{example}

Next, we give a simple example of a quadratic harnesses with $\gamma>1$ and $\sigma\tau>1$; such examples are interesting because most of the general theory developed in \cite{Bryc-Matysiak-Wesolowski-04} does not apply.

  \begin{example}\label{Ex:PB}
 Suppose $(G_t)_{t>0}$ is a gamma process with both  parameters in \eqref{G-d} equal $1$.  Let $\xi$ be an independent random variable with mean $\E(\xi)=\beta>0$ and $\E\xi^2=v^2$.  Let
$$
X_t=\xi G_t\,.
$$
Then  $\E(X_t)=\beta t$. From
$$\Cov(X_s,X_t)=\E(\Cov(X_s,X_t|\xi))+\Cov(\E(X_s|\xi),\E(X_t|\xi)) $$
 we see that for $s\leq t$, $\Cov(X_s,X_t)=s (v^2 t+\beta )$.
Let $(\mathcal{F}_{s,u})$ be the natural past-future filtration associated with $(X_t)$. Consider auxiliary $\sigma$-fields
 $\widetilde{\mathcal{F}}_{s,u}$ generated by $\xi$ and $\{G_t: t\in(0,s]\cup [u,\infty)\}$.
 Then
 $\E(X_t|\widetilde{\mathcal{F}}_{s,u})=\xi \left(\frac{u-t}{u-s}G_s+\frac{t-s}{u-s}G_u\right)=
 \frac{u-t}{u-s}X_s+\frac{t-s}{u-s}X_u$ so
 $\E(X_t|{\mathcal{F}}_{s,u})=
 \frac{u-t}{u-s}X_s+\frac{t-s}{u-s}X_u$.
Similarly, using \eqref{V-gammaG} with $\alpha=1$ we get
\begin{multline*}
\Var(X_t|\widetilde{\mathcal{F}}_{s,u})=\xi^2 \Var(G_t|G_s,G_u)=\xi^2\frac{(u-t)(t-s)}{(u-s+1)(u-s)^2}(G_u-G_s)^2
\\=\frac{(u-t)(t-s)}{(u-s+1)(u-s)^2}(X_u-X_s)^2,
\end{multline*}
so
$$
\Var(X_t|{\mathcal{F}}_{s,u})=\frac{(u-t)(t-s)}{(u-s+1)(u-s)^2}(X_u-X_s)^2.$$
From Theorem \ref{P-chi2QH}  applied with
$$
a=v,\; b=0,\;c=v,\;d=\beta/v\,,
$$ we see that
 \begin{equation*}\label{Z_-}
 Z_t
=v(1-t)X_{\frac{\beta t}{v^2(1-t)}}-\frac{\beta^2}{v} t
\end{equation*}
is a standard quadratic harness on $(0,1)$ with parameters
$$
\eta= \theta=2v/\beta,\; \sigma=\tau=v^2/\beta^2,\;\gamma=1+2\sqrt{\sigma\tau}.
$$
In particular, $\gamma=1+2\sqrt{\sigma\tau}$ and $\sigma\tau=v^4/\beta^4=(\E(\xi^2))^2/(\E(\xi))^4\geq 1$ can be arbitrarily large.

Of course, the distribution of $\xi$ is arbitrary so the higher moments of $Z_t$ are not determined uniquely and may fail to exist. In particular, Conjecture \ref{Conj} does not hold for quadratic harnesses on finite intervals.

\end{example}

\subsection{Matrix notation}\label{Sec:MN}
For calculations, it will be convenient  to parameterize time as a subset of the projective plane, i.e. using  $\un{t}=\wektT{t}{1}$.
 Throughout this section,  letters $s,t,u\in \TT$ are reserved to denote time, and $\un{s}$, $\un{t}$, and also $\un{u}=[u,1]^T$ have this special meaning also when used with subscripts or primed. We also use the convention that  $s\leq t\leq u$.

We rewrite \eqref{LR}
 in  vector  form as
\be\label{LR-Delta}
\E(X_t|\F_{s,u})=\langle \un{t},\un{\Delta}_{s,u}(\X)\rangle \;, \ee where
the components of  $\un{\Delta}_{s,u}(\X)$ are defined by \eqref{DD}.

It follows from \eqref{LR} that admissible  expectations of a harness $\X$ are affine in $t$,  i.e.,
\be\label{mean} \E(X_t)=\langle \un{t},\un{\mu}\rangle\;,\quad t\in \TT, \ee
where $\un{\mu}=[\mu_1,\mu_2]^T$.
Moreover, if $\X$ is a square integrable harness then by \cite[Proposition 2.1]{Bryc-Matysiak-Wesolowski-04} the admissible covariances
are of the form \be\label{cov} \Cov(X_s,X_t)=\langle \un{s},\:\Sigma
\un{t}\rangle \;,\quad s,t\in \TT,\;\; %
s\leq t\;, \ee where
$$\Sigma=\left[\begin{matrix}{c_0} & c_1\\
c_2&{c_3}
\end{matrix}
\right].$$

Note that under our convention $s\leq t$ so $\Sigma$ is not a symmetric matrix; for example,  covariance  $\min\{s,t\}$ is represented by matrix  $\Sigma=\left[\begin{matrix} 0& 1\\
0&0
\end{matrix}
\right]$.
We also remark that if ${c_3}\geq 0$, $c_1> c_2$, and ${c_0}{c_3}> c_2^2$ then the right hand side of \eqref{cov}  indeed defines a positive definite function on $\TT=(0,\infty)$, and that  processes with  $c_1=c_2$ are degenerate in the sense that $X_t$ is a linear combination of $X_s,X_u$.

 Formula \eqref{EQ:q-Var} can be written  in matrix form as
 \be\label{quwa1} {\Var}(X_t|\F_{s,u})=F_{t,s,u}\left(1+\langle \un{\theta},\un{\Delta}_{s,u}\rangle
+\langle \un{\Delta}_{s,u},\Gamma\un{\Delta}_{s,u}\rangle \right)\;, \ee where
\begin{equation}
  \label{Gamma}
  \un{\theta}=\wekt{\theta}{\eta},\;
\Gamma=\left[\begin{matrix} \tau&-1\\ \gamma &\sigma
\end{matrix}
\right]\;.
\end{equation}
Here $\eta,\theta,\sigma,\tau,\gamma$  are constants independent of $s, t, u$.

\begin{remark}\label{Rem-Gamma}
Of course, any matrix
$$
\Gamma=\macierz{\tau}{\Gamma_{12}}{\Gamma_{21}}{\sigma}
$$
with $\Gamma_{12}+\Gamma_{21}=\gamma-1$ gives the same right hand side of \eqref{quwa1}.
 The standard choice of symmetric $\Gamma$ is in fact inconvenient, see Proposition \ref{P-normalization}.
The choice made in \eqref{Gamma}  matches the notation we used in previous papers:
after substituting $q$ for $\gamma$, the resulting parametrization of the conditional variance is identical to   \cite[(2.14)]{Bryc-Matysiak-Wesolowski-04}.
\end{remark}

The non-random constant $F_{t,s,u}$   is determined uniquely by taking the average of both sides of \eqref{quwa1}. According to \cite[(2.15)]{Bryc-Matysiak-Wesolowski-04}, with the choice of $\Gamma$ as in \eqref{Gamma}, formula \eqref{Ftsu}  holds.
\subsection{Transformations of quadratic harnesses}\label{S:TST}

For a non-degenerate affine function $f$, as defined in \eqref{f} and \eqref{Af} with M\"obius transform  $\varphi :\ST\to\TT$,  and for $s<u$ in $\ST$, the transformed $\sigma$-fields are
\begin{equation}\label{F^f}
\calF^f_{s,u}=\begin{cases}
\calF_{\varphi (s),\varphi (u)} & \mbox{ if } \det(A)>0\,,\\
\calF_{\varphi (u),\varphi (s)} & \mbox{ if } \det(A)<0\,.
\end{cases}
\end{equation}
 It is clear that if $\X$ has    linear regressions and quadratic  conditional variances with respect to past-future filtration $(\mathcal{F}_{s,u})$, then $\X^f$ has    linear regressions and  quadratic  conditional variances with respect to the past-future filtration $(\calF^f_{s,u})$.

 The following technical result describes how the parameters of a quadratic harness change
  under the %
  \ftransf .

 \begin{proposition}\label{P-chi2chi}
 Let $\X$ be a harness \eqref{LR} with respect to a past-future filtration $(\mathcal{F}_{s,u})$ on an open interval $\TT$ with  the first two moments given by \eqref{mean} and \eqref{cov}.
 Suppose that
 \be\label{quwa-chi} {\Var}(X_t|\F_{s,u})=F_{t,s,u}\left(\chi+\langle \un{\theta},\un{\Delta}_{s,u}\rangle
+\langle \un{\Delta}_{s,u},\Gamma\un{\Delta}_{s,u}\rangle \right)\;, \ee with non-random $F_{t,s,u}$,
$\chi\in\RR$, $\un{\theta}\in\RR^2$, and arbitrary $2\times 2$ matrix $\Gamma$. Let $f$ be a non-degenerate affine function \eqref{f} such that $\varphi$ is onto  $\TT$.

Then the process  $\widetilde \X:=\X^f$ on $\ST=\varphi^{-1}(\TT)$, see \eqref{xf},
 satisfies  \eqref{mean}, \eqref{cov} with  $\un\mu$ and $\Sigma$ replaced by
\be\label{muf}
\widetilde{\un{\mu}}=A^T\un{\mu}+\un{m}. \ee
\begin{equation}\label{cov transf}
\widetilde \Sigma=\begin{cases}
A^T\Sigma A & \mbox{ if } \det(A)>0\,, \\
A^T\Sigma^T A & \mbox{ if } \det(A)<0\,.
\end{cases}
\end{equation}
With respect to past-future filtration $(\calF^f_{s,u})$, formulas \eqref{LR} and \eqref{quwa-chi} hold for $\widetilde \X$ with
 \begin{equation}
  \label{Gamma A}
  \widetilde\Gamma  =\begin{cases}
A^{-1}\Gamma  (A^{-1})^T & \mbox{ if } \det(A)>0\,, \\
(A^{-1})^T\Gamma A^{-1} & \mbox{ if } \det(A)<0\,,
\end{cases}
\end{equation}
\begin{equation}\label{theta Y}
\widetilde {\un{\theta} }=A^{-1}\un{\theta}  -(\widetilde\Gamma+\widetilde\Gamma^T)\un{m}\,,
  \end{equation}
  and
  \begin{equation}\label{chi2chi}
\widetilde \chi=\chi-\langle  \widetilde{\un{\theta}},\un{m}\rangle - \langle \un{m},\widetilde\Gamma\un{m}\rangle\,.
\end{equation}
\end{proposition}
We remark that
transformation \eqref{cov transf} preserves a product form of the covariance.  That is, suppose that
\begin{equation}\label{cov-prod}
\Cov(X_s,X_t)=(\eps s+\delta)(\phi t+\psi), \; s<t,
\end{equation}
so that
$$\Sigma=
\macierz{ \eps  \phi  }{ \eps  \psi }{
 \delta  \phi  }{ \delta  \psi}.$$ If for book-keeping we write the coefficients of  \eqref{cov-prod} as
 \begin{equation}
   \label{bookkeeping}
   \Theta=\macierz{\eps}{\delta}{\phi}{\psi},
 \end{equation}
 then a calculation based on \eqref{cov transf} shows that for  $\det A>0$ the covariance of $\widetilde\X$ corresponds to
$\widetilde\Theta=\Theta A$.

We postpone the proof of Proposition \ref{P-chi2chi} until
 Section \ref{Sec:proof_P-chi2chi}, so that we can first clarify the role of non-random constant $F_{t,s,u}$.
 The main point is that in the non-degenerate case with $c_1>c_2$,  this constant is determined uniquely by taking the average of both sides of \eqref{quwa-chi}.
  Furthermore, we  explain when $F_{t,s,u}$ is  given by formula  \eqref{Ftsu}.

\begin{proposition}\label{P-normalization}
Suppose a harness $\X$ has
 mean \eqref{mean} and
non-degenerate covariance \eqref{cov} with $c_1>c_2$. If $\X$ has quadratic conditional variance \eqref{quwa-chi} and the off-diagonal entries of matrix $\Gamma$, see Remark \ref{Rem-Gamma}, are chosen so that
 \begin{equation}\label{consistency}
\chi+\langle\un{\theta},\un{\mu}\rangle+\langle\un{\mu},\Gamma\un{\mu}\rangle+{\rm tr}(\Gamma \Sigma^T)=0,
\end{equation}
 then %
 $u(1+s\sigma)+\tau-s\gamma\ne 0$
 and $F_{t,s,u}$ is given by formula \eqref{Ftsu}. Moreover, transformation formulas in Proposition \ref{P-chi2chi} preserve  \eqref{Ftsu}.
\end{proposition}
 Formulas \eqref{EQ:q-Var} and \eqref{Gamma} illustrate the choice of such $\Gamma$.

\section{Proofs}\label{sect:proofs}

\subsection{Proof of Proposition \ref {PEB}}\label{ProofPEB}
Let $N_t=\{y: p(t,y; v,z_v)>0\}\subset M_t$ and denote by $A$ a generic Borel set.
Let
\begin{equation}\label{f1}
f_{s,t}(x,y)=\frac{p(s,x; t,y)p(t,y; v,z_v)}{p(s,x;v,z_v)} 1_{N_s}(x) .
\end{equation}
For any $s<t$ in $[r,v)$ and $x\in N_s$, define probability measure $\nu_{s,x,t}$ by
\begin{equation}
\nu_{s,x,t}(A)%
=\int_A  f_{s,t}(x,y)\pi_t(dy).
\end{equation}
For any $t\in(r,v)$, we let $\nu_t=\nu_{r,z_r,t}$.

To prove that the above probabilities define a Markov process we  verify Chapman-Kolmogorov equations.
We need to show that  for $x\in N_s$,
\begin{equation}\label{CK-EQ}
\int_A f_{s,u}(x,z) \pi_u(dz)= \int_{N_t}  f_{s,t}(x,y) \left(\int_A  f_{t,u}(y,z)\pi_u(dz)  \right)  \pi_t(dy).
\end{equation}
To this end, we use algebraic identity that holds for all $(x,y)\in N_s\times N_t$,
\begin{equation}\label{ff1}
 f_{s,t}(x,y) f_{t,u}(y,z)=\frac{p(s,x; t,y)p(t,y; u,z)p(u,z; v,z_v)}{p(s,x; v,z_v)} .
\end{equation}
 By \eqref{ff1} and Chapman-Kolmogorov equations for process $(Z_t)$, the right hand side of \eqref{CK-EQ} is
\begin{multline}
 \int_{N_t} \left(\int_A  \frac{p(s,x; t,y)p(t,y; u,z)p(u,z; v,z_v)}{p(s,x;v,z_v)} \pi_u(dz)  \right)  \pi_t(dy)
=\int_A f_{s,u}(x,z) \pi_u(dz)
\\-   \int_{N_t^c} \left(\int_A  \frac{p(s,x; t,y)p(t,y; u,z)p(u,z; v,z_v)}{p(s,x;v,z_v)} \pi_u(dz)  \right)  \pi_t(dy).
\end{multline}
So  to end the proof of \eqref{CK-EQ}, it is enough to show that for $y\in N_t^c$, $A\subset N_u\subset M_u$,

\begin{equation}\label{NC}
 \int_A p(t,y; u,z)p(u,z; v,z_v)    \pi_u(dz)=0.
\end{equation}
To see this, note that for  $y\in N_t^c$ we have $p(t,y;v, z_v)=0 $. Therefore,
\begin{multline}
0\leq \int_A p(t,y; u,z)p(u,z; v,z_v)    \pi_u(dz)\\
\leq  \int_{M_u} p(t,y; u,z)p(u,z; v, z_v)    \pi_u(dz)=p(t,y;v, z_v)=0 .
\end{multline}
The same argument with $(s,x)=(r,z_r)$ shows that Chapman-Kolmogorov equations hold for $\nu_t$. %

 Let $(X_t)_{t\in(r,v)}$ be a Markov process with univariate laws $(\nu_t)$ and transition probabilities $(\nu_{s,x,t})$.
 We now verify that $(X_t)$ is a bridge.  Since $\nu_t=\nu_{r,z_r,t}$, assumption \eqref{Brzeg1} implies that $X_t\toP z_r$ as $t\to r^+$, and similarly \eqref{Brzeg2} implies that  $X_t\toP z_v$ as $t\to v^-$.

 Due to Markov property, it suffices to  verify the implication \eqref{BB-cond-lawsZ}$\Rightarrow$\eqref{BB-cond-lawsX} for $s<t<u$. Fix $g\geq 0$. Assumption  \eqref{BB-cond-lawsZ} implies that for any  measurable function $\psi:M_s\times M_u\to [0,\infty)$,
 \begin{multline}\label{CondZ}
 \int_{M_s\times M_t\times M_u}g(y)\psi(x,z) p(s,x;t,y)p(t,y;u,z)\pi_s(dx)\pi_t(dy)\pi_u(dz)
 \\
 = \int_{M_s\times M_u}h(x,z) \psi(x,z)p(s,x;u,z)\pi_s(dx)\pi_u(dz).
\end{multline}
To prove \eqref{BB-cond-lawsX}, it is enough to show that for
  any  measurable $\varphi:N_s\times N_u\to [0,1]$, we have
  $$\E(g(X_t)\varphi(X_s,X_u))=\E(h(X_s,X_u)\varphi(X_s,X_u)),$$
  which is the same as
  \begin{multline}\label{CondY}
    \int_{N_s\times N_t\times N_u} g(y) \varphi(x,z) f_{s,t}(x,y)f_{t,u}(y,z)f_{r,s}(z_r,x)\pi_s(dx)\pi_t(dy)\pi_u(dz)
    \\
    =
       \int_{N_s\times N_u} h(x,z)\varphi(x,z) f_{s,u}(x,z)f_{r,s}(z_r,x)\pi_s(dx)\pi_u(dz)  .
\end{multline}
Since $p(r,z_r;v,z_v)>0$, from \eqref{f1} and \eqref{ff1}  we see that \eqref{CondY} is equivalent to
\begin{multline*}
    \int_{N_s\times N_t\times N_u} g(y) \varphi(x,z)p(r,z_r;s,x)p(u,z;v,z_v) \\
    \times p(s,x;t,y)p(t,y;u,z)\pi_s(dx)\pi_t(dy)\pi_u(dz)
    \\
    =
       \int_{N_s\times N_u} h(x,z) \varphi(x,z) p(r,z_r;s,x)p(u,z;v,z_v) p(s,x;u,z)\pi_s(dx)\pi_u(dz).
\end{multline*}
Using \eqref{NC}, we can enlarge the region of integration on the left hand side to
$N_s\times M_t\times N_u$. Thus, the identity follows from \eqref{CondZ} applied to
$$
\psi(x,z)= \varphi(x,z)p(r,z_r;s,x)p(u,z;v,z_v) 1_{N_s}(x)1_{N_u}(z).
$$

\qed
\subsection{Proof of Propositions \ref{P-chi2chi} and \ref{P-normalization} }\label{Sec:proof_P-chi2chi}

We re-write   formula \eqref{Ftsu}  in matrix notation using a special matrix
 \begin{equation}
   \label{J} J=\left[\begin{matrix}0& 1\\
-1&0
\end{matrix}
\right].
 \end{equation}
 It is easy to see that $J^2=-I$, $J^T=-J$.  For ease of reference we state also
two less obvious properties:  for $A\in GL_2(\RR)$,
\begin{equation}\label{J-props}
A^T J A=\det(A)J\; \mbox{ and }  J^TAJ=\det(A) (A^{-1})^T.
\end{equation}

Formula \eqref{Ftsu} can now be written as
\begin{equation}\label{Ftsu-vec}
F_{t,s,u}=\frac{\langle \un{t}, J\un{u}\rangle \langle \un{s}, J\un{t}\rangle}{\langle \un{s},J^T\Gamma J \un{u}\rangle}\,.
\end{equation}
This formula makes sense for any $2\times 2$ matrix $\Gamma$ as long as the denominator is non-zero.

\begin{lemma}  \label{L: Delta_f}
Let $f$ be a non-degenerate affine function \eqref{f} with M\"obius transform $\varphi $. If $s'=\varphi (s)$, $u'=\varphi (u)$, then
  \begin{equation}
    \label{Delta_f}
    \un{\Delta}_{s,u}(\X^f)=A^T \un{\Delta}_{s',u'}(\X)+\un{m}\;.
  \end{equation}
\end{lemma}
\begin{proof}
Let $g(x,y)=[x,y]A$ denote the linear part of $f$. Since  $\un{\Delta}_{s,u}(a)=\un{m}$ on a linear function  $a(t)=\langle \un{t},\un{m}\rangle$, and
 $\X^f(t)=\X^g(t)+\langle \un{t},\un{m}\rangle$, we have
 $\un{\Delta}_{s,u}(\X^f)= \un{\Delta}_{s,u}(\X^g) + \un{m}$.
Since $X_{s'}=\X^g(s)/(cs+d)$,  and from the
matrix form of \eqref{DD} we have
\begin{equation}
  \label{D matrix}
  \un{\Delta}_{s',u'}(\X)=\frac{J(X_{u'}\un{s}'-X_{s'}\un{u}')}{\langle \un{u}',J\un{s}'\rangle},
\end{equation}
we get
\begin{equation*}
\un{\Delta}_{s',u'}(\X)= J\frac{\frac{\X^g(u)}{cu+d}\un{s}'-\frac{\X^g(s)}{cs+d}\un{u}'}{\langle \un{u}',J\un{s}'\rangle}=
J\frac{\X^g(u)(cs+d)\un{s'}- {\X^g(s)}(cu+d)\un{u}'}{\langle(cu+d) \un{u}',J(cs+d)\un{s}'\rangle}\,.
\end{equation*}

Noting that
\begin{equation}\label{hta}
(cs+d)\un{s}'= A\un{s},
\end{equation}  and using \eqref{J-props} we get
\begin{multline*}
\un{\Delta}_{s',u'}(\X)=
JA\frac{\X^g(u)\un{s}- {\X^g(s)}\un{u}}{\langle \un{u},A^TJA\un{s}\rangle}=
(A^{-1})^{T}A^TJA\frac{\X^g(u)\un{s}- {\X^g(s)}\un{u}}{\langle \un{u},A^TJA\un{s}\rangle}
\\=(A^{-1})^{T}\un{\Delta}_{s,u}(\X^g).
\end{multline*}
Thus $\un{\Delta}_{s,u}(\X^f)= \un{\Delta}_{s,u}(\X^g) + \un{m}=A^T\un{\Delta}_{s',u'}(\X)+ \un{m}$.
\end{proof}
\begin{proof}[Proof of Proposition \ref{P-chi2chi}]
Throughout the proof we write $t'=\varphi (t)$ as in Lemma \ref{L: Delta_f}.
 If $\varphi$ is increasing, by \eqref{LR-Delta} and the
definition of $\widetilde \X=\X^f$ we have \be\label{first}
\E(\widetilde X_t|\F^f_{s,u})=(ct+d)\E(X_{t'}|\F_{s',u'})+\langle \un{t},\un{m}\rangle
\ee
$$=(ct+d)\langle \un{t}',\un{\Delta}_{s',u'}(\X)\rangle +\langle \un{t},\un{m}\rangle \;.
$$
By \eqref{hta} and \eqref{Delta_f} we get
$$
\E(\widetilde X_t|\F^f_{s,u})=\langle
A\un{t},(A^{-1})^T\left(\un{\Delta}_{s,u}(\widetilde\X)-\un{m}\right)\rangle +\langle \un{t},\un{m}\rangle =\langle \un{t},\un{\Delta}_{s,u}(\widetilde\X)\rangle \;.
$$
Thus the condition \eqref{LR} holds true and $\widetilde \X$ is a
harness. Similarly, one can verify that \eqref{LR} holds when $\varphi$ is a decreasing function.

We use \eqref{hta} to compute the mean of $\widetilde \X$
$$
\E(\widetilde X_t)=(ct+d)\langle \un{t}',\un{\mu}\rangle +\langle \un{t},\un{m}\rangle =\langle
A\un{t},\un{\mu}\rangle +\langle \un{t},\un{m}\rangle =\langle \un{t},\un{m}+A^T\un{\mu}\rangle \;,
$$
and \eqref{muf} follows.

To find the covariance we again use \eqref{hta} and the fact that
$\Cov\left(X_{s'},X_{t'}\right)$ is either $\langle \un{s}',\Sigma\un{t}'\rangle$ or $\langle \un{t}',\Sigma \un{s}'\rangle=\langle \un{s}',\Sigma^T\un{t}'\rangle$ depending whether $s'<t'$ (case $\det(A)>0$) or $s'>t'$  (case $\det(A)<0$).
For example, if $\det(A)>0$ then
\begin{multline*}
\Cov(\widetilde X_s,\widetilde X_t)=(cs+d)(ct+d)\Cov\left(X_{s'},X_{t'}\right)\\
=(cs+d)(ct+d)\langle \un{s}',\Sigma\un{t}'\rangle =
\langle  A\un{s},\Sigma A\un{t}\rangle =\langle \un{s},A^T\Sigma A\un{t}\rangle\;,
\end{multline*}
and thus \eqref{cov transf} follows. (We omit the proof when $\det A<0$.)

Next we tackle the conditional variance. Since $\phi$ is monotone on $\ST$,
\begin{equation}
  \label{Q-tranf}
 \Var(\widetilde X_t|\calF_{s,u}^f)=\begin{cases}
  (ct+d)^2  \Var( X_{t'}|\calF_{s',u'}) &\mbox{ if $\det(A)>0$}\,, \\ \\
  (ct+d)^2  \Var( X_{t'}|\calF_{u',s'}) & \mbox{ if $\det(A)<0$}\,.
\end{cases}
\end{equation}

Consider the case   $\det(A)<0$ so that $u'<t'<s'$.  Since $\un{\Delta}_{a,b}=\un{\Delta}_{b,a}$, by \eqref{Q-tranf}  and Lemma \ref{L: Delta_f}, the conditional variance is
\begin{multline}
  \label{VVV} \Var(\widetilde X_t|\calF^f_{s,u})
  =(ct+d)^2F_{t',u',s'}\Big(\chi+\left\langle A^{-1}\un{\theta},\un{\Delta}_{s,u}(\widetilde\X)-\un{m}\right\rangle
\\+\left\langle \un{\Delta}_{s,u}(\widetilde\X)-\un{m},A^{-1}\Gamma(A^{-1})^T(\un{\Delta}_{s,u}(\widetilde\X)-\un{m}) \right\rangle \Big).
\end{multline}

Using \eqref{Ftsu-vec}, \eqref{hta} and \eqref{J-props},
we get
\begin{multline*}(ct+d)^2F_{t',u',s'} =(ct+d)^2\frac{\langle\un{u'},J\un{t'}\rangle\langle\un{t'},J\un{s'}\rangle}{\langle\un{u'},J^T \Gamma  J\un{s'}\rangle} \\
 =\frac{\langle(cu+d)\un{u'},J(ct+d)\un{t'}\rangle\langle(ct+d)\un{t'},J (cs+d)\un{s'}\rangle}{\langle(cu+d)\un{u'},J^T \Gamma  J(cs+d)\un{s'}\rangle}
 \\
 =\frac{\langle A\un{u},JA\un{t}\rangle\langle A\un{t},J A\un{s}\rangle}{\langle A\un{u},J^T \Gamma  JA\un{s}\rangle}
=\frac{\langle\un{s},J\un{t}\rangle\langle\un{t},J\un{u}\rangle}{\langle\un{s},J^T \widetilde\Gamma  J\un{u}\rangle}.
\end{multline*}

So formula \eqref{VVV} rewrites   as
  \begin{multline*}
\Var(\widetilde X_t|\calF^f_{s,u})=(ct+d)^2F_{t',u',s'}\Big(\chi-\langle A^{-1}\un{\theta},\un{m}\rangle+\langle\un{m},\widetilde\Gamma^T\un{m}\rangle
\\
+\left\langle\un{\Delta}_{s,u}(\widetilde \X), A^{-1}\un{\theta}-(\widetilde\Gamma+\widetilde{\Gamma}^T)\un{m}\right\rangle
+\langle\un{\Delta}_{s,u}(\widetilde {\X}),\widetilde\Gamma \un{\Delta}_{s,u}(\widetilde {\X})
\Big)
\\=\frac{\langle\un{s},J\un{t}\rangle\langle\un{t},J\un{u}\rangle}{\langle\un{s},J^T \widetilde\Gamma  J\un{u}\rangle}\Big(
\tilde\chi+\left\langle\un{\Delta}_{s,u}(\widetilde \X), A^{-1}\un{\theta}-(\widetilde\Gamma+\widetilde{\Gamma}^T)\un{m}\right\rangle\\
+\langle\un{\Delta}_{s,u}(\widetilde {\X}),\widetilde\Gamma^T \un{\Delta}_{s,u}(\widetilde {\X})\rangle\Big).
  \end{multline*}
Since the last term is invariant under transposition,   we get \eqref{theta Y} and \eqref{Gamma A}.
 The case
  $\det(A)>0$ is handled similarly and the proof is omitted.
\end{proof}
The proof of Proposition \ref{P-normalization} is based on the formula for the covariance matrix of vector $\un{\Delta}_{s,u}$.
\begin{lemma}
\be\label{covv}
\cov\,\un{\Delta}_{s,u}=\frac{c_1-c_2}{u-s}\,J\un{u}\,\un{s}^TJ^T+\Sigma^T.
\ee
\end{lemma}
\begin{proof}
From \eqref{D matrix} we get
\begin{multline*}
\cov\,\un{\Delta}_{s,u}=\E\left(\un{\Delta}_{s,u}\un{\Delta}_{s,u}^T\right)-\E\left(\un{\Delta}_{s,u}\right)\E\left(\un{\Delta}_{s,u}^T\right)
\\=\frac{J\s\u^T\Sigma\u\s^TJ^T-J\u\s^T\Sigma\u\s^TJ^T+J\u\s^T\Sigma\s\u^tJ^T-J\s\s^T\Sigma\u\u^TJ^T}{(u-s)^2}.
\end{multline*}

Note that since $\s^T\Sigma\u=\u^T\Sigma^T\s$ and $\s\,\u^T-\u\,\s^T=(u-s)J^T$, the numerator can be written as
$$
J(\s\,\u^T-\u\,\s^T)\Sigma\u\,\s^TJ^T-J(\s\,\u^T-\u\,\s^T)\Sigma^T\s\,\u^TJ^T=(u-s)(\Sigma\u\,\s^T-\Sigma^T\s\,\u^T)J^T
$$
(recall that $J^T=-J$ and $JJ^T=I$). Further we write the above expression as
$$
(u-s)[(\Sigma-\Sigma^T)\u\,\s^T+\Sigma^T(\u\,\s^T-\s\,\u^T)]J^T=(u-s)(c_1-c_2)J\u\,\s^TJ^T+(u-s)^2\Sigma^TJJ^T
$$
and thus \eqref{covv} follows.
\end{proof}

\begin{proof}[Proof of Proposition \ref{P-normalization}]
We first remark that formulas  \eqref{cov transf}  and \eqref{Gamma A} imply that
$$
\tr(\Gamma\Sigma^T)=\tr(\widetilde\Gamma\widetilde\Sigma).
$$
Next, we note that
$$\chi+\langle \un{\theta},\un{\mu}\rangle+\langle \un{\mu},\Gamma\un{\mu}\rangle
=\widetilde \chi+\langle \widetilde{\un{\theta}}, \widetilde{\un{\mu}}\rangle+\langle  \widetilde{\un{\mu}}, \widetilde{\Gamma} \widetilde{\un{\mu}}\rangle.
$$
(This follows from a longer calculation based on the formulas from Proposition \ref{P-chi2chi}.)
Therefore, transformation formulas preserve \eqref{consistency}.

Next, we show that \eqref{consistency} implies \eqref{Ftsu-vec}. This will be accomplished by computing the averages of both sides of \eqref{quwa-chi}.

We first note that for any harness with covariance \eqref{cov}, the expected value of the left hand side of \eqref{quwa-chi} is
\begin{equation}\label{E(Var)}
\E\left(\V(X_t|\calF_{s,u})\right)=\frac{(t-s)(u-t)}{u-s}(c_1-c_2).
\end{equation}
To prove  \eqref{E(Var)}, we  use \eqref{covv}.  From \eqref{LR-Delta} we get
\begin{multline*}
\E\,\var(X_t|\mathcal{F}_{s,u})=\var\,X_t- \Var(\E(X_t|\calF_{s,u}))=
\var\,X_t-\t^T\cov(\un{\Delta}_{s,u})\t
\\=\t^T\Sigma\t-\frac{c_1-c_2}{u-s}\t^TJ\u\,\s^TJ^T\t-\t^T\Sigma^T\t
=\frac{c_1-c_2}{u-s}(\s^TJ\t) (\t^TJ\u)
\\
=\frac{c_1-c_2}{u-s}(t-s)(u-s).
\end{multline*}

Next, we compute the right hand side of \eqref{quwa-chi}.
With
$K(\un{\Delta}_{s,u})=\chi+\langle \un{\theta},\un{\Delta}_{s,u}\rangle
+\langle \un{\Delta}_{s,u},\Gamma\un{\Delta}_{s,u}\rangle$, we have
\begin{equation}\label{E(K)}
\E\left(K(\un{\Delta}_{s,u})\right)=
\tr(\Gamma\Sigma^T)+K(\un{\mu})+(c_1-c_2)\frac{\langle \un{s},J^T\Gamma J \un{u}\rangle}{u-s}.
\end{equation}
To prove  \eqref{E(K)} we  note that
$
\E\,\un{\Delta}_{s,u}=\un{\mu}
$, so
\begin{multline}\label{K**}
\E\,K(\un{\Delta}_{s,u})=\chi+\E\,\theta^T\un{\Delta}_{s,u}+\E\left(\un{\Delta}_{s,u}^T\Gamma \un{\Delta}_{s,u}\right)\\ =
\chi+\theta^T\un{\mu}+\un{\mu}^T\Gamma \un{\mu}+\mathrm{tr}\,\left(\Gamma \cov\,\un{\Delta}_{s,u}\right)
=K(\un{\mu})+\tr\left(\Gamma \cov\,\un{\Delta}_{s,u}\right).
\end{multline}
From \eqref{covv} we get
\begin{equation}\label{K***}
\tr\left(\Gamma \cov\,\un{\Delta}_{s,u}\right)=\frac{c_1-c_2}{u-s}\tr\left(\Gamma\,J\u\,\s^TJ^T\right)+\tr\left(\Gamma\Sigma^T\right).
\end{equation}
Since
$$
\tr\left(\Gamma\,J\u\,\s^TJ^T\right)=\tr\left(\s^TJ^T \Gamma J \u\right)=\langle \un{s},J^T\Gamma J \un{u}\rangle,
$$
 \eqref{E(K)} follows from \eqref{K**} and \eqref{K***}.

Since $\tr(\Gamma\Sigma^T)+K(\un{\mu})=0$ by \eqref{consistency},
and  $\E\left(\V(X_t|\calF_{s,u})\right)=F_{t,s,u}\E\left(K(\un{\Delta}_{s,u})\right)$,
 in the non-degenerate  case $c_1>c_2$,
formula \eqref{E(Var)} implies that $\E\left(K(\un{\Delta}_{s,u})\right)\ne 0$ so  from \eqref{E(K)}  we see that $\langle \un{s},J^T\Gamma J \un{u}\rangle=u(1+s\sigma)+\tau-s\gamma\ne 0$. We also see that $F_{t,s,u}$ is given by formula \eqref{Ftsu-vec}, which is just a matrix form of \eqref{Ftsu}.

\end{proof}

\subsection{Proof of Theorem \ref{P-chi2QH}}
Let $A=\frac{1}{ad-bc}\macierz{d}{-b}{-c}{a}$ so that its inverse is $B=\macierz{a}{b}{c}{d}$.
We apply Proposition \ref{P-chi2chi} with $f(x,y)=([x,y]-[\alpha,\beta])B$ to
$$
\un{\mu}=\wekt{\beta}{\alpha},\; \Sigma=\left[
\begin{array}{ll}
 a c & a d \\
 b c & b d
\end{array}
\right],\; \Gamma=\macierz{\tau}{\rho/2}{\rho/2}{\sigma},\; \un{\theta}=\wekt{\theta}{\eta}.
$$
From the transformation formulas we get $\widetilde \mu=0$,  $\widetilde \Sigma=\macierz{0}{1}{0}{0}$, and  $\widetilde \chi=\chi+\alpha\eta+\theta\beta+\sigma \alpha^2+\tau\beta^2+2\rho\alpha\beta>0$  by \eqref{chi+}.
  We also get
  $$\widetilde{\un{\theta}}=\left[
\begin{array}{l}
 b (\eta +\beta  \rho +2 \alpha  \sigma )+a
   (\theta +\alpha  \rho +2 \beta  \tau ) \\
 d (\eta +\beta  \rho +2 \alpha  \sigma )+c
   (\theta +\alpha  \rho +2 \beta  \tau )
\end{array}
\right]$$
  and
  $$\widetilde{\Gamma}=\left[
\begin{array}{ll}
 \tau  a^2+b \rho  a+b^2 \sigma  & \frac{1}{2}
   (b c \rho +a d \rho +2 b d \sigma +2 a c \tau
   ) \\
 \frac{1}{2} (b c \rho +a d \rho +2 b d \sigma
   +2 a c \tau ) & \tau  c^2+d \rho  c+d^2
   \sigma
\end{array}
\right]\,.$$
The quadratic polynomial $K$ remains unchanged if  we replace $\widetilde \Gamma$ by
$$
\Gamma'=\macierz{\tau  a^2+b \rho  a+b^2 \sigma}{-\widetilde{\chi}}{\widetilde{\chi}+b c \rho +a d \rho +2 b d \sigma +2 a c \tau }{ \tau  c^2+d \rho  c+d^2
   \sigma
},
$$
see Remark \ref{Rem-Gamma}.
Rewriting $K$ as
$$K(\un{x})=
\widetilde{\chi}+\langle \widetilde{\un{\theta}},\un{x}\rangle+\langle \un{x},\Gamma'\un{x}\rangle=
\widetilde{\chi}\left( 1+\langle \frac{1}{\widetilde{\chi}}\widetilde{\un{\theta}},\un{x}\rangle+\langle \un{x},\frac{1}{\widetilde{\chi}}\Gamma'\un{x}\rangle\right),
$$
we get the parameters as claimed.

\qed

\subsection{Proof of Theorem \ref{P B2}}\label{Sect:proof_of_Thm1.1}
  With $s<t_1<t_2<u$, the conditional covariance of a standard quadratic harness  is
\begin{equation}
  \label{cond-cov}
\Cov\big(X_{t_1},X_{t_2}\big|\calF_{s,u}\big)=\frac{\left\langle \un{t}_2,J\un{u}\right\rangle\left\langle \un{s},J\un{t}_1\right\rangle}
{\langle \un{s},J^T\Gamma J \un{u}\rangle} K(\un{\Delta}_{s,u}),
\end{equation}
where $K(\un{a})=1+\langle \un{\theta},\un{a}\rangle+\langle\un{a},\Gamma\un{a}\rangle$
is the quadratic polynomial from \eqref{Eq:K} and  \eqref{quwa1}.
A quick way to see this is to notice that \eqref{LR} implies
$$ \Cov\big(X_{t_1},X_{t_2}\big|\calF_{s,u}\big)= \frac{u-t_2}{u-t_1}\Var(X_{t_1}|\calF_{s,u})=
 \frac{\langle\un{t}_2,J\un{u}\rangle}{\langle\un{t}_1,J\un{u}\rangle}\Var(X_{t_1}|\calF_{s,u}).$$

\begin{proof}[Proof of Theorem \ref{P B2}]

To prove that $v(1+r\sigma)+\tau-r\gamma>0$ we note that  $v\mapsto v(1+r\sigma)+\tau-r\gamma$ is a continuous function on $\RR$ which by
 Proposition \ref{P-normalization} applied to $\Z$
  cannot  cross zero on $(r,\infty)$.

To determine the mean and the variance of $\X$, we use Definition \ref{Def-Bridge}. Denote by $\calF_{s,u}^X$   the natural past-future filtration of $\X$.
By taking $g(x)=x$, from Definition \ref{Def-Bridge}(iii)  we see that for $t\in(r,v)$, we have
$
\E(X_t|\calF_{s,u}^X)=h_{s,t,u}(X_s,X_u)
$,
where
$$
h_{s,t,u}(x,y)=\frac{u-t}{u-s}x+\frac{t-s}{u-s}y.
$$
Thus by Definition \ref{Def-Bridge}(i),
$\E(X_t|\calF_{s,u}^X)\toP h_{r,t,v}(z_r,z_v)$ %
as $(s,u)\to (r,v)$.

 We note that as $s\searrow r$ and $u\nearrow v$, say over rational numbers, the  filtration
$ \calF_{s,u}^X$ decrease. So by the martingale convergence theorem,
$$\E(X_t)=\E\left(\lim_{(s,u)\to (r^+,v^-)}\E\left(X_t|\calF_{s,u}^X\right)\right)
=\frac{v-t}{v-r}z_r+\frac{t-r}{v-r}z_v=\langle \un{\Delta}_{r,v},\un{t}\rangle\,.
$$

Similarly, by martingale convergence theorem,  for fixed $t_1<t_2$ in $(r,v)$,
$$\E(X_{t_1}X_{t_2})=\E\left(\lim_{(s,u)\to (r^+,v^-)}\E\left(X_{t_1}X_{t_2}|\calF_{s,u}^X\right)\right)$$ %
Using again Definition \ref{Def-Bridge} with $g:\RR^2\to\RR$ defined by $g(x_1,x_2)=x_1x_2$, we see that
$$
\E\left(X_{t_1}X_{t_2}|\calF_{s,u}^X\right)=\tilde h_{s,u}(X_s,X_u),
$$
where $\tilde h$, defined through
$$
\E\left(Z_{t_1}Z_{t_2}|\calF_{s,u}\right)=\tilde h_{s,u}(Z_s,Z_u),
$$
 is computed from \eqref{cond-cov} as
$$
\tilde h_{s,u}(x,y)=\tfrac{(u-t_2)(t_1-s)}{u(1+s\sigma)+\tau-s\gamma } K\left(\tfrac{y-x}{u-s},\tfrac{ux-sy}{u-s}\right) +h_{s,t_1,u}(x,y) h_{s,t_2,u}(x,y).
$$
Since $\tilde h_{s,u}(X_s,X_u)\toP \tilde h_{r,v}(z_r,z_v)$   as $(s,u)\to (r^+,v^-)$, we get
 $\Cov(X_{t_1},X_{t_2})= M^2(v-t_2)(t_1-r)$ with
 \begin{equation}
  \label{MMM}
  M=\frac{\sqrt{K(\Delta_{r,v},\widetilde\Delta_{r,v})}}{\sqrt{v(1+r\sigma)+\tau-r\gamma}}>0\,.
\end{equation}

The above reasoning also shows that for $r<s<t<u<v$ the conditional  moments
 $\E(X_t|\calF_{s,u}^X)$ and $\V(X_t|\calF_{s,u}^X)$ %
 are given by the same polynomials as the corresponding conditional moments of $\Z$.
Thus the assumptions of Theorem \ref{P-chi2QH} hold, and we apply it with
$$
\alpha=\widetilde\Delta_{r,v},\; \beta=\Delta_{r,v},\; \chi=1, \; \rho=\gamma-1,
$$
$$
a=M \sqrt{v},\; b=-rM\sqrt{v},\; c=-M/\sqrt{v},\;d=M\sqrt{v}.
$$
(There are other possible choices that lead to "equivalent" quadratic  harnesses as in \eqref{EQ:equiv}.)
With the above choice of $a,b,c,d$,   formula \eqref{EQ:chi2QH} gives \eqref{AM-bridge}.
Since $\widetilde \chi=K(\Delta_{r,v},\widetilde\Delta_{r,v})>0$, assumption \eqref{chi+} holds,
and the parameters of the resulting quadratic harness are as claimed.

 \end{proof}

\subsection*{Acknowledgement}
Maja Jamio\l kowska worked out the construction \cite{Maja:2009}  which inspired our Proposition \ref{P: bi-poisson}. We also benefited from discussions with Ryszard Szwarc and Wojciech Matysiak.
This research was partially supported by NSF
grant \#DMS-0904720, and by  Taft Research Seminar 2008/09.

\bibliographystyle{acm}
\bibliography{../Vita,Mobius_08}

\end{document}